\documentclass[12pt]{article}
\usepackage{amsmath, amsthm}
\usepackage{amsfonts}
\usepackage{amssymb}
\usepackage[usenames,dvipsnames]{color}
\usepackage{graphicx}
\usepackage{hyperref}

\newcommand{\R}{\ensuremath{\mathbb{R}}}

\newcommand{\C}{\ensuremath{\mathcal{C}}}

\newcommand{\E}{\ensuremath{\mathbb{E}}}

\newcommand{\W}[2]{\ensuremath{W^{(#1)}(#2)}}
\newcommand{\dW}[2]{\ensuremath{W^{(#1)\prime}(#2)}}

\newcommand{\abs}[1]{\ensuremath{\left|#1\right|}}
\newcommand{\ZZ}[3]{\ensuremath{Z_{#3}^{(#1)}(#2)}}
\newcommand{\WW}[3]{\ensuremath{W_{#3}^{(#1)}(#2)}}
\newcommand{\dWW}[3]{\ensuremath{W_{#3}^{(#1)\prime}(#2)}}
\renewcommand{\P}{\ensuremath{\mathbb{P}}}

\newtheorem{theorem}{Theorem}[section]
\newtheorem{lem}[theorem]{Lemma}
\newtheorem{prop}[theorem]{Proposition}
\newtheorem{cor}[theorem]{Corollary}
\newtheorem{rem}[theorem]{Remark}

\title{Optimal prediction for positive self-similar Markov processes}

\author{Erik Baurdoux\thanks{Department of Statistics, LSE,
Houghton Street,
London, WC2A 2AE, UK. Email: \texttt{e.j.baurdoux@lse.ac.uk}},
\, Andreas E. Kyprianou\thanks{Corresponding author. Department of Mathematical Sciences, University of Bath, Claverton Down, Bath, BA2 7AY, UK. Email: \texttt{a.kyprianou@bath.ac.uk}} 
\, and 
Curdin Ott\thanks{Email: \texttt{curdin.ott@bluewin.ch}}
}

\begin{document}

\maketitle

\begin{abstract}
This paper addresses the question of predicting when a positive self-similar Markov process $X$ attains its pathwise global supremum or infimum before hitting zero for the first time (if it does at all). This problem has been studied in~\cite{GloHulPes} under the assumption that $X$ is a positive transient diffusion. We extend their result to the class of positive self-similar Markov processes by establishing a link to~\cite{baur_schaik}, where the same question is studied for a L\'evy process drifting to $-\infty$. The connection to~\cite{baur_schaik}  relies on the so-called Lamperti transformation ~\cite{lamperti} which links the class of positive self-similar Markov processes with that of L\'evy processes. 
Our approach will reveal that the results in~\cite{GloHulPes} for Bessel processes can also be seen as a consequence of self-similarity.

\medskip


\medskip

\end{abstract}

\section{Introduction}\label{c5_introduction}
In keeping with the development of a family of prediction problems for Brownian motion and, more generally, L\'evy processes, cf. \cite{GPS, dTP, GloHulPes, baur_schaik} to name but a few, we address the question of predicting the time when a positive self-similar Markov process (pssMp) attains its pathwise global supremum or infimum. Aside from the embedding of this problem within the general theory of optimal stopping, the interest and novelty in the current setting is to show that, in contrast to the approach in~\cite{GloHulPes} for self-similar diffusions,  the problem can reduced via time-change to a more homogenous setting.

We shall spend some time to set up some notation in order to formulate the problem rigorously.
A positive self-similar Markov process $X=\{X_t:t\geq 0\}$ with self-similarity index $\alpha>0$ is a $[0,\infty)$-valued standard Markov process defined on a filtered probability space \mbox{$(\Omega,\mathcal{G}, \mathbb{G}:=\{\mathcal{G}_t:t\geq 0\},\{P_x:x>0\}$)}, which has $0$ as an absorbing state and which satisfies the scaling property: for every $x,c>0$,
\begin{equation*}
\text{the law of }\{cX_{c^{-\alpha}t}:t\geq 0\}\text{ under $P_x$ is equal to the law of $X$ under $P_{cx}$}.
\end{equation*}
Here, we mean ``standard'' in the sense that $\mathbb{G}$ satisfies the natural conditions (cf.~\cite{bichteler}, Section 1.3, page 39) and $X$ is strong Markov with c\`adl\`ag and quasi-left-continuous paths. Lamperti~\cite{lamperti} proved in a seminal paper that the family of pssMp splits into three exhaustive classes which can be distinguished from each other by comparing their hitting time of $0$, that is, $\zeta:=\inf\{t>0:X_t=0\}$. The classification reads as follows:
{\renewcommand{\theenumi}{\roman{enumi}}
\renewcommand{\labelenumi}{(\theenumi)}
\begin{enumerate}
\item\label{c5_type1} $P_x[\zeta=\infty]=1$ for all starting points $x>0$,
\item\label{c5_type2} $P_x[\zeta<\infty, X_{\zeta-}=0]=1$ for all starting points $x>0$,
\item\label{c5_type3} $P_x[\zeta<\infty, X_{\zeta-}>0]=1$ for all starting points $x>0$.
\end{enumerate}}
\noindent In other words, a pssMp $X$ starting at $x>0$ either never hits zero, hits zero continuously or hits zero by jumping onto it. The two subclasses of pssMps that are used here are
\begin{eqnarray*}
\mathcal{C}&:=&\big\{X\text{ is spectrally negative with non-monotone paths and}\\
&&\;\;\text{either of type~\eqref{c5_type2} or~\eqref{c5_type3}}\big\},\label{c5_def_class}\\
\hat{\mathcal{C}}&:=&\big\{X\text{ is spectrally positive with non-monotone paths and}\\
&&\;\;\text{either of type~\eqref{c5_type1} and drifting to $\infty$ or of type~\eqref{c5_type3}}\big\}.\notag
\end{eqnarray*}
By spectrally negative and spectrally positive we mean that the trajectories of $X$ only have downward or upward jumps respectively.\\
\indent One of the aims here is to answer the following question: Given $X\in\mathcal{C}$, is it possible to stop ``as close as possible'' to the time at which $X$ ``attains'' its supremum? In more mathematical terms, define
\begin{equation*}
\Theta:=\sup\{t\geq 0:X_t=\overline X_\zeta\}=\sup\{0\leq t<\zeta:X_t=\overline X_\zeta\},
\end{equation*}
where $\overline X=\{\overline X_t:t\geq 0\}$ is the running maximum process $\overline X_t:=\sup_{0\leq u\leq t}X_u$, $t\geq 0$. By definition of $\mathcal{C}$, it follows that the set $\{t\geq 0: X_t=\overline X_\zeta\}$ is a singleton; see Subsection~\ref{c5_Lamperti} for details. We are interested in the optimal stopping problem 
\begin{equation}
\inf_{\tau}E_x[\vert\Theta-\tau\vert-\Theta],\label{c5_problem1}
\end{equation}
where $x>0$ and the infimum is taken over a certain set of $\mathbb{G}$-stopping times $\tau$ which is specified later. The term ``attains'' is used in a loose sense here. Indeed, if $X$ has negative jumps it might happen that the supremum is never attained. However, the above definition ensures that we have $X_{\Theta}=\overline X_\zeta$ on the event $\{X_\Theta\geq X_{\Theta -}\}$ while $X_{\Theta -}=\overline X_\zeta$ on the event $\{X_\Theta< X_{\Theta -}\}$.\\
\indent Analogously, one may try to stop ``as close as possible'' to the time at which a process $X\in\hat{\mathcal{C}}$ ``attains'' its infimum before hitting zero (if at all). To this end, let
\begin{equation*}
\hat\Theta:=\sup\{0\leq t<\zeta:X_t=\underline X_{\zeta-}\},
\end{equation*}
where $\underline X=\{\underline X_t:t\geq 0\}$ the running minimum process $\underline X_t:=\inf_{0\leq u\leq t}X_u$, $t\geq 0$. Again, by definition of $\hat{\mathcal{C}}$, the set $\{0\leq t<\zeta:X_t=\underline X_{\zeta-}\}$ a singleton; see Subsection~\ref{c5_Lamperti} for details. If $X$ has positive jumps, the word ``attains'' is used in a loose sense analogously to above. Stopping as close as possible to~$\hat\Theta$ then leads to solving the optimal stopping problem
\begin{equation}
\inf_{\tau}E_x[\vert\hat\Theta-\tau\vert-\hat\Theta],\label{c5_problem2}
\end{equation}
where $x>0$ and the infimum is taken over a certain set of $\mathbb{G}$-stopping times $\tau$ which is specified later.\\
\indent Our interest in~\eqref{c5_problem1} and~\eqref{c5_problem2} was raised thanks  to~\cite{GloHulPes}, where the authors solve~\eqref{c5_problem2} under the assumption that $X$ is a diffusion in $(0,\infty)$ such that $\lim_{t\to\infty}X_t=\infty$. Their result states that the optimal stopping time is given by
\begin{equation}
\rho^*_1=\inf\{t\geq 0:X_t\geq f^*(\underline X_t)\},\label{c5_intro_stopping_time}
\end{equation}
where $f^*$ is the minimal solution to a certain differential equation. In particular, when $X$ is a $d$-dimensional Bessel process with $d>2$, it is shown that $f^*(z)=\lambda^*_1 z$, $z\geq 0$, for some constant $\lambda_1^*>1$, which is a root of some polynomial. Due to the fact that the class of Bessel processes for $d>2$ belongs to the class of pssMps with $\alpha=2$, it is possible to express the optimal stopping time~\eqref{c5_intro_stopping_time} (up to a time-change) in terms of the underlying Lamperti representation $\xi$ (of $X$) reflected at its infimum. This raises the suspicion that the simple form of~\eqref{c5_intro_stopping_time} in the Bessel case could be a consequence of the self-similarity of $X$ and suggests that~\eqref{c5_problem2} (or an analogue of it) can also be solved for the class of pssMps.\\
\indent In this paper we show that the speculations in the previous paragraph are indeed true. Specifically, we prove that the optimal stopping times in~\eqref{c5_problem1} and~\eqref{c5_problem2} are of the simple form
\begin{equation*}
\tau^*=\inf\{t\geq 0:X_t\geq K^*\overline X_t\}\quad\text{and}\quad\hat \tau^*=\inf\{t\geq 0:X_t\leq\hat K^*\underline X_t\}
\end{equation*} 
for some constants $0<K^*<1$ and $\hat K^*>1$ respectively. As alluded to above, the key step is to reduce~\eqref{c5_problem1} and~\eqref{c5_problem2} to a one-dimensional problem with the help of the so-called Lamperti transformation~\cite{lamperti} which links pssMps to L\'evy processes.\\
\section{Preliminaries}
\subsection{Killed L\'evy processes}
A process $\xi$ with values in $\R\cup\{-\infty\}$ is called a L\'evy process killed at rate $q\geq 0$ if $\xi$ starts at $0$, has stationary and independent increments and $\mathbf{k}:=\inf\{t>0:\xi_t=-\infty\}$ has an exponential distribution with parameter $q\geq 0$. In the case $q=0$ it is understood that $\P[\mathbf{k}=\infty]=1$, that is, no killing. It is well known that a L\'evy process $X$ killed at rate $q$ is characterised by its L\'evy triplet $(\gamma,\sigma,\Pi)$ and the killing rate~$q$, where $\sigma\geq0, \gamma\in\R$ and $\Pi$ is a measure on $\R$ satisfying the condition $\int_{\R}(1\wedge x^2)\,\Pi({\rm d}x)<\infty$. The Laplace exponent of $\xi$ under $\P$ is defined by
\begin{equation*}
\psi(\theta):=\log(\E[{\rm e}^{\theta \xi_1}])
\end{equation*}
for any $\theta\in\R$ such that $\psi(\theta)<\infty$. It is known that (cf.~Theorem 3.6 in~\cite{kyprianou}), \mbox{for $\theta\in\R$},
\begin{equation}
\E[{\rm e}^{\theta \xi_t}]<\infty\text{ for all $t\geq 0$}\qquad\Longleftrightarrow\qquad \int_{\abs{x}\geq1}{\rm e}^{\theta x}\,\Pi({\rm d}x)<\infty,\label{c5_crit}
\end{equation}
and in this case we have
\begin{equation}
\psi(\theta)=-q-\gamma\theta+\frac{1}{2}\sigma^2\theta^2+\int_{\R}\big({\rm e}^{\theta x}-1-\theta x1_{\{\abs{x}<1\}}\big)\,\Pi({\rm d}x).\label{c5_l_exponent}
\end{equation}
In particular, if $\xi$ is of bounded variation,~\eqref{c5_l_exponent} may be written as
\begin{equation*}
\psi(\theta)=-q +\mathtt{d}\theta-\int_{\R}(1-{\rm e}^{\theta x})\,\Pi({\rm d}x)
\end{equation*} 
for some $\mathtt{d}\in\R$.\\
\indent Finally, for any killed L\'evy process (starting at zero) and any $v\in\R$ with $\psi(v)<\infty$ the process
\begin{equation*}
\exp(v\xi_t-\psi(v)t)1_{\{t<\mathbf{k}\}},\qquad t\geq 0,
\end{equation*}
is a $\P$-martingale. Hence, we may further define the family of measures $\{\P^v\}$ with Ra\-don-Ni\-ko\-dym derivatives
\begin{equation}
\frac{{\rm d}\P^v}{{\rm d}\P}\bigg\vert_{\mathcal{F}_t}=\exp(v\xi_t-\psi(v)t)1_{\{t<\mathbf{k}\}},\label{c5_changeofmeasure}
\end{equation}
where $\{\mathcal{F}_t : t\geq 0\}$ is the natural filtration associated to $\xi$.
In particular, under $\P^v$ the process $\xi$ is a L\'evy process and its Laplace exponent is given by $\psi_v(\theta)=\psi(v+\theta)-\psi(v)$ and infinite lifetime, that is, $\P^v[\mathbf{k}=\infty]=1$; cf.~Theorem 3.9 in~\cite{kyprianou}.

\subsection{Scale functions}
We suppose throughout this subsection that $\xi$ is an unkilled spectrally negative L\'evy process ($q=0$). Spectrally negative means that $\Pi$ is concentrated on $(-\infty,0)$ and thus $\xi$ only exhibits downward jumps. Observe that in this case, the Laplace exponent $\psi(\theta)$ exists at least for $\theta\geq 0$ by~\eqref{c5_crit}. Its right-inverse is defined by
\begin{equation*}
\Phi(\lambda):=\sup\{\theta\geq 0:\psi(\theta)=\lambda\},\qquad\lambda\geq 0.
\end{equation*} 
 A special family of functions associated with unkilled spectrally negative L\'evy processes is that of scale functions (cf.~\cite{KuzKypRiv,kyprianou}) which are defined as follows. For $\eta\geq 0$, the $\eta$-scale function $W^{(\eta)}:\R\to[0,\infty)$ is the unique function whose restriction to $(0,\infty)$ is continuous and has Laplace transform
\begin{equation*}
\int_0^\infty {\rm e}^{-\theta x}\W{\eta}{x}\,{\rm d}x=\frac{1}{\psi(\theta)-\eta},\quad\theta>\Phi(\eta),
\end{equation*}
and is defined to be identically zero for $x\leq 0$. Further, we shall use the notation $\WW{\eta}{x}{v}$ to mean the $\eta$-scale function associated to $X$ under $\P^v$. For fixed $x\geq 0$, it is also possible to analytically extend $\eta\mapsto\W{\eta}{x}$ to $\eta\in\mathbb{C}$. A useful relation that links the different scale functions is (cf. Lemma 3.7 in~\cite{KuzKypRiv})
\begin{equation}
\W{\eta}{x}={\rm e}^{vx}\WW{\eta-\psi(v)}{x}{v}\label{c5_scale5}
\end{equation}
for $v\in\R$ such that $\psi(v)<\infty$ and $\eta\in\mathbb{C}$. Moreover, the following regularity properties of scale functions are known; cf. Sections 2.3 and 3.1 of~\cite{KuzKypRiv}.\\

\noindent\textsl{Smoothness}: For all $\eta\geq 0$, $W^{(\eta)}$ is Lebesgue-almost everywhere differentiable. Moreover, 
\begin{equation}\label{c5_smoothness}
W^{(\eta)}\vert_{(0,\infty)}\in\begin{cases}
C^1(0,\infty),&\text{if $X$ is of bounded variation and $\Pi$ has no atoms},\\C^1(0,\infty),&\text{if $X$ is of unbounded variation and $\sigma=0$},\\C^2(0,\infty),&\text{$\sigma>0$}.
\end{cases}
\end{equation}
\\
\textsl{Continuity at the origin:} For all $\eta\geq 0$,
\begin{equation}\label{c5_continuityatorigin}
\W{q}{0+}=\begin{cases}\mathtt{d}^{-1},&\text{if $X$ is of bounded variation,}\\0,&\text{if $X$ is of unbounded variation.}
\end{cases}
\end{equation}
\\
\textsl{Right-derivative at the origin:} For all $q\geq 0$,
\begin{equation}\label{c5_derivativeatorigin}
W^{(q)\prime}_+(0+)=\begin{cases}
\frac{q+\Pi(-\infty,0)}{\mathtt{d}^2},&\text{if $\sigma=0$ and $\Pi(-\infty,0)<\infty$,}\\
\frac{2}{\sigma^2},&\text{if $\sigma>0$ or $\Pi(-\infty,0)=\infty$,}\
\end{cases}
\end{equation}
where we understand the second case to be $+\infty$ when $\sigma=0$.

The second scale function is $Z_v^{(\eta)}$ and defined as follows. For $v\in\R$ such that $\psi(v)<\infty$ and $\eta\geq 0$ we define $Z_v^{(\eta)}:\R\longrightarrow[1,\infty)$ by
\begin{equation}
\ZZ{\eta}{x}{v}=1+\eta\int_0^x\WW{\eta}{z}{v}\,{\rm d}z.\label{c5_def_Z}
\end{equation}

\subsection{The Lamperti transformation}\label{c5_Lamperti}
Lamperti's main result in~\cite{lamperti} asserts that any pssMp $X$ may, up to its first hitting time of zero, be expressed as the exponential of a time-changed L\'evy process. We will now explain this in more detail. Instead of writing $(X,P_x)$ to denote the positive self-similar Markov process starting at $x>0$, we shall sometimes write $X^{(x)}=\{X_t^{(x)}:t\geq 0\}$  in order to emphasise the dependency of the path on its initial value. Similarly, we write $\zeta^{(x)}=\inf\{t>0:X_t^{(x)}=0\}$.

For fixed $x>0$ define
\begin{equation*}
\varphi(t):=\int_0^{x^\alpha t}(X^{(x)}_s)^{-\alpha}\,{\rm d}s,\qquad t<x^{-\alpha}\zeta^{(x)}.
\end{equation*}
It will be important to understand the behaviour of $\varphi(x^{-\alpha}\zeta-):=\lim_{t\uparrow\zeta}\varphi(x^{-\alpha}t)$. In particular, note that the distribution of $\varphi(x^{-\alpha}\zeta-)$ does not depend on $x>0$. Moreover, the following result is known; see Lemma 13.3 in~\cite{kyprianou}.
\begin{lem}
In the case that $\zeta=\infty$ or that $\{\zeta<\infty\text{ and }X_{\zeta -}=0\}$, we have $P_x[\varphi(x^{-\alpha}\zeta-)=\infty]=1$, for all $x>0$. In the case that $\zeta<\infty$ and $X_{\zeta-}>0$, we have that, under $P_x$, $\varphi(x^{-\alpha}\zeta-)$ is exponentially distributed with a parameter that does not depend on the value of $x>0$.
\end{lem}
As the distribution of $\varphi(x^{-\alpha}\zeta^{(x)}-)$ is independent of $x$, we will rename it~$\mathbf{e}$. When $\mathbf{e}=\infty$ almost surely we interpret it as an exponential distribution with parameter zero. Now define the right-inverse of $\varphi$,
\begin{equation*}
I_u:=\inf\{0<t<x^{-\alpha}\zeta^{(x)}:\varphi(t)>u\},\quad u\geq 0.
\end{equation*}
Moreover, define the process $\xi:=\{\xi_t:t\geq 0\}$ by setting, for $x>0$,
\begin{equation*}
\xi_t:=\log(X_{x^\alpha I_t}/x),\qquad 0\leq t<\mathbf{e}
\end{equation*}
and $\xi_t=-\infty$ for $t\geq\mathbf{e}$ (in the case that $\mathbf{e}<\infty$).
The main result in~\cite{lamperti} states that a pssMp is nothing else than a space and time-changed killed L\'evy process.
\begin{prop}[Lamperti transformation]\label{c5_Lamp}
If $X^{(x)}$, $x>0$, is a positive self-similar Markov process with index of self-similarity $\alpha>0$, then it can be represented as
\begin{equation*}
X^{(x)}_t=x\exp(\xi_{\varphi(x^{-\alpha}t)}),\qquad t\geq 0,
\end{equation*}
and either
{\renewcommand{\theenumi}{\roman{enumi}}
\renewcommand{\labelenumi}{(\theenumi)}
\begin{enumerate}
\item $\zeta^{(x)}=\infty$ almost surely for all $x>0$, in which case $\xi$ is an unkilled L\'evy process satisfying $\limsup_{t\uparrow\infty}\xi_t=\infty$, or
\item $\zeta^{(x)}<\infty, X^{(x)}_{\zeta^{(x)}-}=0$ almost surely for all $x>0$, in which case $\xi$ is an unkilled L\'evy process satisfying $\lim_{t\uparrow\infty}\xi_t=-\infty$, or
\item $\zeta^{(x)}<\infty, X^{(x)}_{\zeta^{(x)}-}>0$ almost surely for all $x>0$, in which case $\xi$ is a killed L\'evy process.
\end{enumerate}}
Also note that we may identify
\begin{equation*}
I_t=\int_0^t{\rm e}^{\alpha\xi_s}\,{\rm d}s,\qquad t<\mathbf{e}.
\end{equation*}
\end{prop}
The version of the Lamperti transformation we have just given is Theorem~13.1 in~\cite{kyprianou}, where one can also find a proof of it.\\
\indent We conclude this subsection by explaining why the sets $\{t\geq 0:X_t=\overline X_\zeta\}$ and $\{0\leq t<\zeta:X_t=\underline X_{\zeta-}\}$ mentioned in the introduction are singletons. By definition of $\mathcal{C}$ and $\hat{\mathcal{C}}$ it is clear that both sets are non-empty, but they could potentially contain more than one element. In view of the Lamperti transformation we see that the aforementioned sets contain only a single element provided the same is true for the sets $\{t\geq 0:\xi_t=\sup_{0\leq u<\infty}\xi_u\}$ and $\{0\leq t<\mathbf{e}:\xi_t=\inf_{0\leq u<t}\xi_u\}$, where $\xi$ is the underlying Lamperti representation of $X$ in $\mathcal{C}$ and $\hat{\mathcal{C}}$ respectively. However, it is known that local extrema (and hence global extrema) of L\'evy processes are distinct except for compound Poisson processes, see Proposition 4 in~\cite{bertoin_book}. But for $X$ in $\mathcal{C}$ or $\hat{\mathcal{C}}$ the Lamperti transformation can never be a compound Poisson process and thus the assertion follows.

\section{Reformulation of problems and main results}\label{c5_refor_maxi_mini}
\subsection{Predicting the time at which the maximum is attained}\label{c5_maxi}
Suppose throughout this subsection that $X\in\mathcal{C}$ with parameter of self-similarity $\alpha>0$ and let $\xi$ be its Lamperti representation which is a spectrally negative L\'evy process killed at some rate $q\geq0$ satisfying $\lim_{t\uparrow\infty}\xi_t=-\infty$ whenever $q=0$. For $\theta\geq 0$, let $\psi(\theta)$ be the Laplace exponent of $\xi$ and $\phi(\theta)=q+\psi(\theta)$ the Laplace exponent of $\xi$ unkilled. Denote by $\Phi$ the right-inverse of $\phi$ and note that $\Phi(q)>0$.\\
\indent We begin our analysis with two steps that are almost identical to Lemmas~1 and 2 of~\cite{GloHulPes}. For this reason, we omit the proof of the first of the two lemmas below and streamline the proof the second to the particular case at hand.
\begin{lem}\label{c5_rewrite_0}
For any $\mathbb{G}$-stopping time $\tau$ we have
\begin{equation*}
\abs{\Theta-\tau}=\Theta+\int_0^\tau\big(21_{\{\Theta\leq t\}}-1\big)\,{\rm d}t.
\end{equation*}
\end{lem}

\begin{lem}\label{c5_rewrite_1}
For $x>0$ and any $\mathbb{G}$-stopping time $\tau$ with finite mean we~have
\begin{equation}
E_x[\vert\Theta-\tau\vert-\Theta]=E_x\big[\int_0^{\tau\wedge\zeta}F(\overline X_t/X_t)\,{\rm d}t]+E_x[(\tau-\zeta)1_{\{\tau>\zeta\}}]\label{c5_rewrite_2},
\end{equation}
where $F(y)=1-2y^{-\Phi(q)}$, $y\geq 1$.
\end{lem}
\begin{proof}
For any $\mathbb{G}$-stopping time $\tau$ with finite mean, following verbatim the proof of Lemma 2 in \cite{GloHulPes}, we have by Fubini's theorem,
\begin{eqnarray}
E_x\big[\int_0^\tau (21_{\{\Theta\leq t\}}-1)\,{\rm d}t\big]
&=&E_x\big[\int_0^\tau (1-2\P_x[\Theta>t\vert\mathcal{G}_t])1_{\{t<\zeta\}}\,{\rm d}t\big]\label{c5_refo_1}\\
&&+E_x[(\tau-\zeta)1_{\{\tau>\zeta\}}].\notag
\end{eqnarray}
Using the strong Markov property of $X$ we obtain on $\{t<\zeta\}$,
\begin{eqnarray*}
P_x[\Theta>t\vert\mathcal{G}_t]&=&P_x\Big[\overline X_t<\sup_{t\leq u<\zeta}X_u\Big\vert\mathcal{G}_t\Big]\\
&=&P_x\Big[s<\sup_{t\leq u<\zeta}X_u\Big\vert\mathcal{G}_t\Big]\Big\vert_{s=\overline X_t}\\
&=&P_{x'}\Big[s<\sup_{0\leq u<\zeta}X_u\Big]\Big\vert_{s=\overline X_t, x' = X_t}.
\end{eqnarray*}
Hence, using the Lamperti transformation we obtain for $0<x\leq s$,
\begin{equation*}
P_x\Big[s<\sup_{0\leq u<\zeta}X_u\Big]=P_x\Big[\log(s/x)<\sup_{0\leq u<\mathbf{e}}\xi_u\Big]={\rm e}^{-\Phi(q)\log(s/x)}.
\end{equation*}
Plugging this into~\eqref{c5_refo_1} gives the result.
\end{proof}

We are interested in minimising the expectation on the left-hand side of~\eqref{c5_rewrite_2} over the set $\mathcal{M}$ of all integrable $\mathbb{G}$-stopping times $\tau$. The requirement that $\tau$ is integrable ensures that~\eqref{c5_rewrite_2} is well defined. Taking into account the specific form of the right-hand side of~\eqref{c5_rewrite_2}, one sees that for $x>0$,
\begin{equation*}
\inf_{\tau\in\mathcal{M}} E_x[\vert\Theta-\tau\vert-\Theta]=\inf_{\tau\in\mathcal{M}} E_x\big[\int_0^{\tau\wedge\zeta}F(\overline X_t/X_t)\,{\rm d}t].
\end{equation*}


It turns out that, in providing a solution to (\ref{c5_problem1}) we need to restrict ourselves to the case that the underlying L\'evy process in the Lamperti transform (and hence the pssMp) satisfies a condition. We therefore define the modified class
\begin{eqnarray*}
\mathcal{C}^1&:=&\{\text{$X\in\mathcal{C}$ such that {\color{black}$\psi(\alpha)<0$}}\}.
\end{eqnarray*}
The criterion $\psi(\alpha)<0$ is a technical one, which turns out to be equivalent to $\Theta$ being finite in mean (see Theorem \ref{meantheta}). Later on, in  Section~\ref{c5_examples},  we will provide examples where this condition can be checked.
\indent Summing up, for $X\in\mathcal{C}^1$ we are led to the optimal stopping problem
\begin{equation}
v(x,s)=\inf_{\tau} E_x\big[\int_0^{\tau\wedge\zeta} F((s\vee\overline X_t)/X_t)\,{\rm d}t],\qquad 0<x\leq s,\label{c5_problem3}
\end{equation}
where the infimum is taken over all integrable $\mathbb{G}$-stopping times $\tau$. We are now in a position to state our first main result.
 
\begin{theorem}\label{c5_main_res_1}
Let $X\in\mathcal{C}^1$ with index of self-similarity $\alpha>0$, in which case its Lamperti representation $\xi$ is a spectrally negative L\'evy process killed at rate $q\geq 0$. 
Recall that $\phi$ is the Laplace exponent of $\xi$ unkilled and $\Phi$ its right-inverse. Let $\W{\cdot}{z}$ be the scale function associated with $\phi$. Then the solution of~\eqref{c5_problem3} is given by
\begin{equation*}
v(x,s)=-\int_{K^*s}^xz^{\alpha-1}\bigg(1-2\Big(\frac{z}{s}\Big)^{\Phi(q)}\bigg)\W{q}{\log(x/z)}\,{\rm d}z
\end{equation*}
and $\tau^*:=\inf\{t\geq0:X_t\leq K^*(s\vee\overline X_t)\}$, where $K^*\in(0,2^{-\frac{1}{\Phi(q)}})$ is the unique solution to the equation (in $K$)
\begin{equation}
\int_0^{\log(1/K)}(1-2{\rm e}^{-\Phi(q)z})\dWW{q-\phi(\alpha)}{z}{\alpha}\,{\rm d}z=\W{q}{0}\quad\text{on $(0,1)$.}\label{c5_for_k_1}
\end{equation} 
\end{theorem}

\begin{rem}\label{c5_tech_con}
\rm
The right-hand side of~\eqref{c5_for_k_1} is equal to zero unless $\xi$ is of bounded variation; see~\eqref{c5_continuityatorigin}.
\end{rem}

It is interesting to note that 
if $E_x(\Theta)<\infty$, then minimising $E_x[\vert\Theta-\tau\vert-\Theta]$ is equivalent to minimising $E_x[\vert\Theta-\tau\vert]$.  The following theorem, also proved in Section \ref{c5_proofs}, shows that the additional condition on $\xi$ in the restricted class $\mathcal{C}^1$ implies that $\Theta$ always has a finite mean.

\begin{theorem}\label{meantheta}
 When $X$ belongs to $\mathcal{C}$, we have $E_x(\Theta)<\infty$ for all $x>0$ if and only if $\psi(\alpha)<0$.
\end{theorem}

\noindent Noting from the Lamperti transformation  that $\zeta =\int_0^{\mathbf{e}}{\rm e}^{\alpha\xi_t}{\rm d}t$, one also sees that 
$\psi(\alpha)<0$ is also necessary and sufficient for $E_x(\zeta)<0$ for all $x>0$. 

Theorem \ref{c5_main_res_1} is a result is a consequence of the analysis in Section~\ref{c5_reduction} and~\ref{c5_key_problem} and its proof is given in Section \ref{c5_proofs}.  An explicit example is provided in Section~\ref{c5_examples}. 

\subsection{Predicting the time at which the minimum is attained}\label{c5_mini}
Suppose throughout this subsection that $X\in \hat{\mathcal{C}}$ with parameter of self-similarity $\alpha>0$ and let $\xi$ again be its Lamperti representation which is a spectrally positive L\'evy process killed at rate $q\geq 0$ satisfying $\lim_{t\uparrow\infty}\xi_t=\infty$ whenever $q=0$. Introduce the dual $\hat\xi=\{\hat\xi_t:t\geq 0\}$ of $\xi$ which is defined as
\begin{equation*}
\hat\xi_t:=\begin{cases}-\xi_t,&t<\mathbf{e},\\
-\infty,&t\geq\bf{e},\end{cases}
\end{equation*}
where $\mathbf{e}=\inf\{t>0:\xi_t=-\infty\}$. It follows that $\hat\xi$ is a spectrally negative L\'evy process killed at rate $q\geq 0$ satisfying $\lim_{t\uparrow\infty}\hat\xi_t=-\infty$ whenever $q=0$. For $\theta\geq 0$, let $\hat\psi$ be the Laplace exponent of $\hat\xi$ and $\hat\phi(\theta)=q+\hat\psi(\theta)$ the Laplace exponent of $\hat\xi$ unkilled. Finally, denote by $\hat\Phi$ the right-inverse of $\hat\phi$ and note that $\hat\Phi(q)>0$.\\
\indent Analogously to Lemma~\ref{c5_rewrite_0} and~\ref{c5_rewrite_1}, one can prove the following result.
\begin{lem}
For $x>0$ and any $\mathbb{G}$-stopping time $\tau$ with finite mean we~have
\begin{equation}
E_x[\vert\hat\Theta-\tau\vert-\hat\Theta]=E_x\big[\int_0^{\tau\wedge\zeta}\hat F(X_t/\underline X_t)\,{\rm d}t]+E_x[(\tau-\zeta)1_{\{\tau>\zeta\}}]\label{c5_rewrite_3},
\end{equation}
where $\hat F(y):=1-2y^{-\hat\Phi(q)}$, $y\geq 1$.
\end{lem}
The specific form of the right-hand side of~\eqref{c5_rewrite_3} shows again that for $x>0$,
\begin{equation*}
\inf_{\tau\in\mathcal{M}}E_x[\vert\hat\Theta-\tau\vert-\hat\Theta]=\inf_{\tau\in\mathcal{M}}E_x\big[\int_0^{\tau\wedge\zeta}\hat F(X_t/\underline X_t)\,{\rm d}t],
\end{equation*}
where $\mathcal{M}$ is the set of all integrable $\mathbb{G}$-stopping times $\tau$. 
Similarly to the problem of predicting the maximum, in order to solve the problem of predicting the minimum, we need to work in a more restrictive class of pssMp than $\hat{\mathcal{C}}.$ To this end, let use define

\begin{eqnarray*}
\hat{\mathcal{C}}^1&:=&\{\text{$X\in\hat{\mathcal{C}}$ such that $\hat\psi$ exists at $-\alpha$ and $\hat\psi(-\alpha)<0$ if $q>0$}\}.
\end{eqnarray*}

For $X\in\hat{\mathcal{C}}^1$, we are led to the optimal stopping problem
\begin{equation}
\hat v(x,i):=\inf_{\tau}E_x[\int_0^{\tau\wedge\zeta}\hat F(X_t/(i\wedge\underline X_t))\,{\rm d}t],\qquad 0<i\leq x,\label{c5_problem4}
\end{equation}
where the infimum is taken respectively with the two cases over all $\mathbb{G}$-stopping times $\tau$ or all integrable $\mathbb{G}$-stopping times $\tau$. We can now state the analogue of Theorem~\ref{c5_main_res_1}. 

\begin{theorem}\label{c5_main_res_2}
Assume that $X\in\hat{\mathcal{C}}^1$ with index of self-similarity $\alpha>0$, in which case the dual $\hat \xi$ of the Lamperti representation of $X$ is a spectrally negative L\'evy process killed at rate $q\geq 0$. 
Moreover, recall that $\hat\phi$ is the Laplace exponent of the dual $\hat\xi$ unkilled and $\hat\Phi$ its right-inverse. Let $\hat W^{(\cdot)}(z)$ be the scale function associated with $\hat\phi$. Then the solution of~\eqref{c5_problem4} is given by
\begin{equation*}
\hat v(x,i)=-\int_x^{\hat K^*i}z^{\alpha-1}\bigg(1-2\Big(\frac{i}{z}\Big)^{\hat\Phi(q)}\bigg)\hat W^{(q)}(\log(z/x))\,{\rm d}z
\end{equation*}
and $\hat\tau^*:=\inf\{t\geq0:X_t\geq\hat K^*(i\wedge\underline X_t)\}$, where $\hat K^*>2^{1/\hat\Phi(q)}$ is the unique solution to the equation (in $K$) 
\begin{equation}
\int_0^{\log(K)}(1-2{\rm e}^{-\hat\Phi(q)z})\hat W_{-\alpha}^{(q-\hat\phi(-\alpha))\prime}(z)\,{\rm d}z=\hat W^{(q)}(0)\qquad\text{on $(1,\infty)$.}\label{c5_for_k_2}
\end{equation}
\end{theorem}

\noindent This result is again a consequence of the analysis of Sections~\ref{c5_reduction} and~\ref{c5_key_problem} and the analogue of Remark~\ref{c5_tech_con} applies here as well. An example including the case when $X$ is a $d$-dimensional Bessel process for $d>2$ is provided in Section~\ref{c5_examples}.

As in Subsection~\ref{c5_maxi}, it is natural to 
ask when $\hat\Theta$ has a finite mean. In this respect we have the following result.

\begin{theorem}\label{thetahatfinite} Suppose that $X\in\hat{\mathcal{C}}$ and $q=0$. Then $E_x(\hat\Theta)<\infty$, for all $x>0$, if and only if $\hat\psi(\alpha)<0$. If $X\in\hat{\mathcal{C}}$ and $q>0$ then $\hat\psi(\alpha)<0$ becomes a sufficient condition.
\end{theorem}

\noindent Note also that, if $q=0$ and $X\in \hat{\mathcal{C}}$, then the issue of whether $E_x(\zeta)<\infty$ is irrelevant since $\zeta=\infty$ almost surely. On the other hand, when $q>0$ and $X\in\hat{\mathcal{C}}$, again noting that $\zeta = \int_0^{\mathbf{e}}{\rm e}^\alpha\xi_t{\rm d}t$, we see that $E_x(\zeta)<\infty$ for all $x>0$ if and only if $\psi(-\alpha)<0$, in which case one also has $E_x(\hat\Theta)<\infty$ for all $x>0$.

\section{Reduction to a one-dimensional problem}\label{c5_reduction}
\subsection{Reduction of problem~\eqref{c5_problem3}}\label{c5_redu_1}
The aim in this subsection is to reduce~\eqref{c5_problem3} to a one-dimensional optimal stopping problem.\\
\indent We begin by reducing~\eqref{c5_problem3} to an optimal stopping problem in which $X$ starts at $x=1$. More precisely, the self-similarity of $X$ implies that the process
 \begin{equation}
 \int_0^{t\wedge\zeta^{(x)}}F((s\vee\overline X^{(x)}_u)/X^{(x)}_u)\,{\rm d}u,\qquad t\geq 0,\label{c5_pro1}
 \end{equation}
is equal in law to the process
\begin{equation}
x^\alpha\int_0^{(x^{-\alpha}t)\wedge\zeta^{(1)}} F(((s/x)\vee \overline X^{(1)}_u)/X^{(1)}_u)\,{\rm d}u,\qquad t\geq 0.\label{c5_pro2}
\end{equation}
Note that the process in~\eqref{c5_pro1} is adapted to $\mathbb{G}$, whereas the process in~\eqref{c5_pro2} is adapted to $\tilde{\mathbb{G}}^{(x)}=\{\tilde{\mathcal{G}}^{(x)}_u:u\geq 0\}$, where $\tilde{\mathcal{G}}^{(x)}_u:=\mathcal{G}_{x^{-\alpha}u}$. 
We conclude that for $0<x\leq s$,
\begin{eqnarray*}
v(x,s)&=&\inf_\tau E_x[\int_0^{\tau\wedge\zeta} F((s\vee\overline X_t)/X_t)\,{\rm d}t]\\
&=&x^\alpha\inf_{\tau^\prime} E_1[\int_0^{(x^{-\alpha}\tau^\prime)\wedge\zeta} F(((s/x)\vee\overline X_t)/X_t)\,{\rm d}t],
\end{eqnarray*}
where the first infimum is taken over $\mathbb{G}$-stopping times $\tau$ and the second over $\tilde{\mathbb{G}}^{(x)}$-stopping times $\tau^\prime$. Before we can continue with the reduction of~\eqref{c5_problem3}, we need to introduce a new filtration $\mathbb{H}:=\{\mathcal{H}_t:t\geq 0\}$ in $\mathcal{G}$. Recall that the process
\begin{equation*}
\varphi(t)=\int_0^t(X^{(1)}_u)^{-\alpha}\,{\rm d}u,\qquad t<\zeta^{(1)},
\end{equation*}
is right-continuous and adapted to~$\mathbb{G}$. Then
\begin{equation*}
I_u=\inf\{0<t<\zeta^{(1)}:\varphi(t)>u\},\qquad u\geq 0,
\end{equation*}
is a right-continuous process which is strictly increasing on $[0,\varphi(\zeta^{(1)}-))$. In particular, $I_u$ is a $\mathbb{G}$-stopping time for each $u\geq 0$. We now use $I_u$, $u\geq 0$, to time-change the filtration $\mathbb{G}$ according to
\begin{equation}
\mathcal{H}_u:=\mathcal{G}_{I_u},\qquad u\geq 0.\label{c5_def_H}
\end{equation}
By Lemma 7.3 in~\cite{kallenberg} it follows that $\mathbb{H}$ is right-continuous. Also observe that the Lamperti representation $\xi$ is adapted to $\mathbb{H}$. Finally, denote by $\mathcal{M}_1^{(x)}$ the set of all $\tilde{\mathbb{G}}^{(x)}$-stopping times and by $\mathcal{M}_2$ the set of all $\mathbb{H}$-stopping times. As a final piece of notation before we formulate the main result of this subsection, define the measure $P^\alpha$ by
\begin{eqnarray}
\frac{{\rm d}P^\alpha}{{\rm d}P_1}\bigg\vert_{\mathcal{H}_t}={\rm e}^{\alpha\xi_t-\psi(\alpha)t}1_{\{t<\mathbf{e}\}}.\label{c5_change1}
\end{eqnarray}

\begin{lem}\label{c5_reduction_1}
Let $f(z)=1-2{\rm e}^{-\Phi(q)z}$, $z\geq 0$, where $\Phi$ and $q$ are as at the beginning of Subsection~\ref{c5_maxi}. For $0<x\leq s$, we have
\begin{eqnarray}
v(x,s)&=&x^{\alpha}\inf_{\tau^\prime\in\mathcal{M}^{(x)}_1}E_1[\int_0^{(x^{-\alpha}\tau^\prime)\wedge\zeta} F(((s/x)\vee\overline X_t)/X_t)\,{\rm d}t]\label{c5_problem5}\\*
&\geq&x^{\alpha}\inf_{\nu\in\mathcal{M}_2}E^\alpha[\int_0^{\nu}{\rm e}^{\psi(\alpha)u}f(Y^{\log(y)}_u)\,{\rm d}u],\label{c5_one_dim_1}
\end{eqnarray}
where $y=s/x$, $Y^{\log(y)}_u:=\log(y)\vee\overline\xi_u-\xi_u$ and $\overline\xi_u:=\sup_{0\leq t\leq u}\xi_t$ for $u\geq 0$. In particular, under $P^\alpha$ the spectrally negative L\'evy process $\xi$ is not killed.
\end{lem}

\begin{proof}
Using the fact that $\varphi$ is strictly increasing on $[0,\zeta)$ and the Lamperti transformation shows that for $\tau^\prime\in\mathcal{M}_1^{(x)}$,
 \begin{eqnarray}
&&E_1[\int_0^{(x^{-\alpha}\tau^\prime)\wedge\zeta}F((y\vee\overline X_t)/X_t)\,{\rm d}t]\label{c5_change_variable}\\
&&=E_1[\int_0^{(x^{-\alpha}\tau^\prime)\wedge\zeta}F((y\vee\overline X_t)/X_t)1_{\{t<\zeta\}}\,{\rm d}t]\notag\\
&&=E_1\big[\int_0^{(x^{-\alpha}\tau^\prime)\wedge\zeta} f\big(\log(y)\vee\overline \xi_{\varphi(t)}-\xi_{\varphi(t)}\big)1_{\{\varphi(t)<\varphi(\zeta)\}}\,{\rm d}t\big].\notag
\end{eqnarray}
Next, note that $
\varphi^\prime(t)=(X^{(1)}_t)^{-\alpha}={\rm e}^{-\alpha\xi_{\varphi(t)}}$ for $t<\zeta^{(1)}$. Hence, changing variables with $u=\varphi(t)$ shows that the right-hand side of~\eqref{c5_change_variable} is equal to
\begin{equation*}
E_1[\int_0^{\varphi((x^{-\alpha}\tau^\prime)\wedge\zeta)} {\rm e}^{\alpha\xi_u}f\big(\log(y)\vee\overline \xi_u-\xi_u\big)1_{\{u<\mathbf{e}\}}\,{\rm d}u].
\end{equation*}
As $\tau^\prime\in\mathcal{M}_1^{(x)}$, it follows that $\varphi((x^{-\alpha}\tau^\prime)\wedge\zeta)$ is a $\mathbb{H}$-stopping time that is less or equal than $\mathbf{e}$, and hence we conclude that
\begin{equation}
v(x,s)\geq x^{-\alpha}\inf_{\nu\in\mathcal{M}_2}E_1[\int_0^\nu {\rm e}^{\alpha\xi_u}f\big(\log(y)\vee\overline \xi_u-\xi_u\big)1_{\{u<\mathbf{e}\}}\,{\rm d}u].\label{c5_alt_expr}
\end{equation}
In other words, we have found a lower bound for $v(x,s)$ in terms of an optimal stopping problem for the Lamperti representation $\xi$ reflected at its maximum. Using Fubini's theorem and a change of measure according to~\eqref{c5_change1} yields for $\nu\in\mathcal{M}_2$,
\begin{eqnarray*}
&&E_1[\int_0^{\nu}{\rm e}^{\alpha\xi_u}f\big(\log(y)\vee\overline \xi_u-\xi_u\big)1_{\{u<\mathbf{e}\}}\,{\rm d}u]\notag\\
&&=\int_0^\infty E_1[{\rm e}^{\alpha\xi_u}f\big(\log(y)\vee\overline \xi_u-\xi_u\big)1_{\{u<\mathbf{e}\}}1_{\{u<\nu\}}]\,{\rm d}u\label{c5_switch}\\
&&=\int_0^\infty E^\alpha[{\rm e}^{\psi(\alpha)u}f\big(\log(y)\vee\overline \xi_u-\xi_u\big)1_{\{u<\nu\}}]\,{\rm d}u\notag\\
&&=E^\alpha[\int_0^{\nu}{\rm e}^{\psi(\alpha)u}f(Y^{\log(y)}_u)\,{\rm d}u].\notag
\end{eqnarray*}
Finally, note that the Laplace exponent of $\xi$ under $P^\alpha$ is given by the expression $\psi_\alpha(\theta)=\psi(\theta+\alpha)-\psi(\alpha)$, $\theta\geq 0$. In particular, $\psi_\alpha(0)=0$ and hence $\xi$ is not killed under $P^\alpha$.
\end{proof}

Despite the inequality in~\eqref{c5_one_dim_1}, we are in a good enough position with this lemma to deduce the solution of~\eqref{c5_problem3}. To see why, suppose that the optimal stopping time for~\eqref{c5_one_dim_1} is given~by
\begin{equation*}
\nu^*=\inf\{t\geq 0:Y_t^{\log(y)}\geq k^*\}
\end{equation*}
for some $k^*>0$. Additionally, setting $K^*:={\rm e}^{-k^*}$, define 
\begin{eqnarray*}
&&\tau^*=\inf\{t\geq 0:X_t\leq K^*(s\vee\overline X_t)\},\\
&&\tau^\prime=\inf\{t\geq 0:X_{x^{-\alpha}t}\leq K^*((s/x)\vee\overline X_{x^{-\alpha}t})\}.
\end{eqnarray*}
It then holds that
\begin{eqnarray*}
E_x\int_0^{\tau^*} F((s\vee\overline X_t)/X_t)\,{\rm d}t]&=&x^{\alpha}E_1[\int_0^{x^{-\alpha}\tau^\prime}F(((s/x)\vee\overline X_t)/X_t)\,{\rm d}t]\\
&=&x^{\alpha}E^\alpha[\int_0^{\nu^*}{\rm e}^{\psi(\alpha)t}f(Y^{\log(s/x)}_t)\,{\rm d}t]
\end{eqnarray*}
and thus $\tau^*$ is optimal for~\eqref{c5_problem3}. Hence it remains to show that the optimal stopping time for~\eqref{c5_one_dim_1} is indeed of the assumed form. This is done in Section~\ref{c5_key_problem}.

\subsection{Reduction of problem~\eqref{c5_problem4}}\label{c5_redu_2}
Analogously to the previous subsection, we want to reduce~\eqref{c5_problem4} to a one-dimensional optimal stopping problem.\\
\indent Let $\mathcal{M}^{(x)}_1$ be the set of all $\tilde{\mathbb{G}}^{(x)}$-stopping times and 
$\mathcal{M}_2$ the set of all $\mathbb{H}$-stopping times whenever $X\in\hat{\mathcal{C}}^1$ is of 
type~\eqref{c5_type3}. On the other hand, if $X\in\hat{\mathcal{C}}^1$ is of type~\eqref{c5_type1}, then 
denote by $\mathcal{M}^{(x)}_1$ the set of all integrable $\tilde{\mathbb{G}}^{(x)}$-stopping times 
and by $\mathcal{M}_2$ the set of all $\mathbb{H}$-stopping times $\nu$ such that
\begin{equation*}
\hat E^{-\alpha}[\int_0^\nu {\rm e}^{\hat\psi(-\alpha)t}\,{\rm d}t]<\infty,
\end{equation*}
where 
\begin{eqnarray}
\frac{{\rm d}\hat P^{-\alpha}}{{\rm d}P_1}\bigg\vert_{\mathcal{H}_t}={\rm e}^{-\alpha\hat\xi_t-\hat\psi(-\alpha)t}1_{\{t<\mathbf{e}\}}.\label{c5_change2}
\end{eqnarray}

 Following the same line of reasoning as in Subsection~\ref{c5_redu_1}, one may obtain the analogue of Lemma~\ref{c5_reduction_1}; see Lemma~\ref{c5_reduction_2} below. The only difference is that we express all in terms of the dual process $\hat\xi$ so that we obtain a one-dimensional optimal stopping problem in~\eqref{c5_one_dim_2} that is of the same type as in~\eqref{c5_one_dim_1} (a one-dimensional optimal stopping problem for a spectrally negative L\'evy process reflected at its supremum). The advantage of this is that once the one-dimensional problem is solved, we can deduce the solution for both~\eqref{c5_problem3} and~\eqref{c5_problem4}. Moreover, the fact that~\eqref{c5_one_dim_1} and~\eqref{c5_one_dim_2} only differ by switching to the dual essentially says that the problem of predicting the time at which the maximum or minimum is attained is, at least on the level of Lamperti representations, essentially the same.
\begin{lem}\label{c5_reduction_2}
Let $\hat f(z)=1-2{\rm e}^{-\hat\Phi(q)z}$, $z\geq 0$, where $\hat\Phi$ and $q$ are as at the beginning of Subsection~\ref{c5_mini}. 
For $0<i\leq $x, we have
\begin{eqnarray}
\hat v(x,i)&=&x^{\alpha}\inf_{\tau^\prime\in\mathcal{M}^{(x)}_1}E_1[\int_0^{x^{-\alpha}\tau^\prime\wedge\zeta}\hat F(X_t/(i\wedge\underline X_t))\,{\rm d}t]\label{c5_problem6}\\
&\geq&x^\alpha\inf_{\nu\in\mathcal{M}_2}\hat E^{-\alpha}[\int_0^{\nu} {\rm e}^{\hat\psi(-\alpha)u}\hat f(\hat Y^{\log(\hat y)}_t)\,{\rm d}u],\label{c5_one_dim_2}
\end{eqnarray}
where $\hat y=x/i$, $\hat Y^{\log(y)}_u:=\log(y)\vee\overline{\hat\xi_u}-\hat\xi_u$ and $\overline{\hat\xi_u}:=\sup_{0\leq t\leq u}\hat\xi_t$ for $u\geq 0$. In particular, under $\hat P^{-\alpha}$ the spectrally negative L\'evy process $\hat\xi$ is not killed.
\end{lem}
Analogously to Subsection~\ref{c5_redu_1}, it follows that if the optimal stopping time for~\eqref{c5_one_dim_2} is given by $\nu^*=\inf\{t\geq0:Y_t^{\log(y)}\geq \hat{k}^*\}$ for some $\hat{k^*}>0$, then
\begin{equation*}
\hat\tau^*=\inf\{t\geq0:X_t\geq\hat K^*(i\wedge\underline X_t)\}
\end{equation*}
is optimal in~\eqref{c5_problem4}, where $\hat K^*:={\rm e}^{\hat k^*}$. The remaining task is again to solve~\eqref{c5_one_dim_2} and show that the optimal stopping time is indeed given by $\nu^*$. This is done in Section~\ref{c5_key_problem}.

\section{The one-dimensional optimal stopping problem}\label{c5_key_problem}
In this section we solve a separate optimal stopping problem which is set up in such a way that once it is solved one can use it to deduce the solution of~\eqref{c5_one_dim_1} and~\eqref{c5_one_dim_2} and hence the solution of~\eqref{c5_problem3} and~\eqref{c5_problem4} respectively. This section is self-contained and can be read completely independently of Sections~\ref{c5_refor_maxi_mini} and~\ref{c5_reduction}. Therefore, for convenience we will reuse some of the notation -- there should be no confusion.
\subsection{Setting and formulation of one-dimensional problem}
\indent Let us spend some time introducing the notation and formulating the problem. Suppose that $\Xi=\{\Xi_t:t\geq 0\}$ is an (unkilled) spectrally negative L\'evy process defined on a filtered probability space $(\Omega,\mathcal{F},\mathbb{F}:=\{\mathcal{F}_t:t\geq 0\},\tilde\P)$ satisfying the natural conditions; cf.~\cite{bichteler}, Section 1.3, p.39. For convenience we will assume without loss of generality that $(\Omega,\mathcal{F})=(\R^{[0,\infty)},\mathcal{B}^{[0,\infty)})$, where $\mathcal{B}$ is the Borel-$\sigma$-field on~$\R$. The coordinate process on $(\Omega,\mathcal{F})$ is denoted by $Y=\{Y_t:t\geq 0\}$. Further, let $q\geq 0$ and suppose that $\Xi$ under $\tilde\P$ is such that $\lim_{t\uparrow\infty}\Xi_t=-\infty$ whenever $q=0$. Also assume that the L\'evy measure associated with $\Xi$ has no atoms whenever $\Xi$ is of bounded variation. This is a purely technical condition which ensures that the $q$-scale functions $W^{(q)}$ associated with $\Xi$ are continuously differentiable on $(0,\infty)$; see~\eqref{c5_smoothness}. Next, let $\beta\in\R\setminus\{0\}$ such that $\tilde\E[{\rm e}^{\beta\Xi_1}]<\infty$. This condition is automatically satisfied if $\beta>0$ due to the spectral negativity of~$\Xi$ and hence it is only an additional assumption when $\beta<0$. The Laplace exponent is given by
\begin{equation*}
\phi(\theta):=\log(\tilde\E[{\rm e}^{\theta\Xi_1}]),\qquad\theta\geq 0\wedge\beta,
\end{equation*}
and its right-inverse is defined as
\begin{equation*}
\Phi(\lambda):=\sup\{\theta\geq 0:\phi(\theta)=\lambda\},\qquad\lambda\geq 0.
\end{equation*}
In particular, note that $\Phi(q)>0$ and define 
\begin{equation*}
f(y):=1-2{\rm e}^{-\Phi(q)y},\qquad y\geq 0.
\end{equation*}
Moreover, denote by $\tilde\P^\beta$ the measure obtained by the change of measure
\begin{equation*}
\frac{{\rm d}\tilde\P^\beta}{{\rm d}\tilde\P}\bigg\vert_{\mathcal{F}_t}={\rm e}^{\beta\Xi_t-\phi(\beta)t},\qquad t\geq 0.
\end{equation*}
Finally, for $y\geq 0$, let $\P^\beta_y$ be the law of
\begin{equation*}
y\vee\sup_{0\leq u\leq t}\Xi_u-\Xi_t,\qquad t\geq 0,
\end{equation*}
under $\tilde\P^\beta$.\\
\indent We are interested in the optimal stopping problem
\begin{equation}
V^*(y):=\inf_{\tau\in\mathcal{M}}\E_y^\beta[\int_0^\tau {\rm e}^{-qt+\phi(\beta)t}f(Y_t)\,{\rm d}t]\label{c5_problem7}
\end{equation}
for $y\geq 0$ and $(q,\beta)\in\mathcal{A}$, where
\begin{equation*}
\mathcal{A}:=\{(q,\beta)\in[0,\infty)\times\R\setminus\{0\}: q>\phi(\beta)\text{ or $q=0$ and $\beta<0$}\},
\end{equation*}
and the set $\mathcal{M}$ denotes the set of $\mathbb{F}$-stopping times such that
\begin{equation}
\E_y^\beta[\int_0^\tau {\rm e}^{-qt+\phi(\beta)t}\,{\rm d}t]<\infty.\label{c5_int_cond}
\end{equation}
Note that $\mathcal{M}$ is the set of all $\mathbb{F}$-stopping times except when $q=0$ and $\beta<0$ in which case~\eqref{c5_int_cond} is indeed a restriction because $\phi(\beta)>0$ due to the assumption that $\lim_{t\uparrow\infty}\Xi_t=-\infty$.

\subsection{Solution of one-dimensional problem}
\indent Given the underlying Markovian structure of~\eqref{c5_problem7}, it is reasonable to look for an optimal stopping time of the form
\begin{equation*}
\tau_k=\inf\{t\geq0:Y_t\geq k\},\qquad k>0.
\end{equation*}
However, when $q=0$ and $\beta<0$, we need to check whether $\tau_k\in\mathcal{M}$.
\begin{lem}\label{c5_adm}
Let $k>0$. If $q=0$ and $\beta<0$ (and hence $\phi(\beta)>0$), it holds that $\E_y^\beta[\int_0^{\tau_k}{\rm e}^{\phi(\beta)t}\,{\rm d}t]<\infty$ for all $y\geq 0$.
\end{lem}

\begin{proof}
Throughout this proof, let $\overline\Xi_t:=\sup_{0\leq u\leq t}\Xi_u$, $t\geq 0$, and write $\tau_{k,y}:=\inf\{t\geq0:y\vee\overline\Xi_t-\Xi_t\geq k\}$ for $y\geq 0$. If $y\geq k$ the assertion is clearly true and hence suppose that $y<k$. Using the fact that $\beta<0$ in the second inequality, we have
\begin{eqnarray*}
\E_y^\beta[\int_0^{\tau_k}{\rm e}^{\phi(\beta)t}\,{\rm d}t]&=&\E_y^\beta\bigg[\frac{{\rm e}^{\tau_k\phi(\beta)}}{\phi(\beta)}\bigg]-\frac{1}{\phi(\beta)}\\
&\leq&\phi(\beta)^{-1}\E_y^\beta[{\rm e}^{\tau_k\phi(\beta)}]\\
&=&\phi(\beta)^{-1}\tilde\E[{\rm e}^{\beta\Xi_{\tau_{k,y}}}]\\
&=&\phi(\beta)^{-1}\tilde\E\big[{\rm e}^{-\beta(y\vee\overline\Xi_{\tau_{k,y}}-\Xi_{\tau_k,y})+\beta(y\vee\overline\Xi_{\tau_{k,y}})}\big]\\
&\leq&\phi(\beta)^{-1}\tilde\E\big[{\rm e}^{-\beta(y\vee\overline\Xi_{\tau_{k,y}}-\Xi_{\tau_{k,y}})}\big].
\end{eqnarray*}
It is now shown in Theorem 1 in~\cite{exitproblems} that the expression on the right-hand side is finite.
\end{proof}

The next question we address is what the value function associated with the stopping times $\tau_k$ looks like. To this end, introduce the quantity
\begin{equation*}
V_k(y):=\E_y^\beta[\int_0^{\tau_k} {\rm e}^{-qt+\phi(\beta)t}f(Y_t)\,{\rm d}t],\qquad y\geq 0.
\end{equation*}
The next result gives an expression for $V_k$ in terms of scale functions.
\begin{lem}\label{c5_v_k}
For $k\geq 0$, we have
\begin{eqnarray}
V_k(y)&=&-\int_y^kf(z)\WW{q-\phi(\beta)}{z-y}{\beta}\,{\rm d}z\notag\\
&&+\frac{\WW{q-\phi(\beta)}{k-y}{\beta}}{\dWW{q-\psi(\beta)}{k}{\beta}}\bigg(\int_0^kf(z)\dWW{q-\phi(\beta)}{z}{\beta}\,{\rm d}z-\WW{q-\phi(\beta)}{0}{\beta}\bigg).\label{c5_value_functions}
\end{eqnarray}
\end{lem}

\begin{proof}
Define for $\eta\geq 0$ the functions
\begin{equation*}
\tilde V_k(y):=\E_y^\beta[\int_0^{\tau_k} {\rm e}^{-\eta t}f(Y_t)\,{\rm d}t].
\end{equation*}
Now recall from Theorem 8.11 in~\cite{kyprianou} that the density of the $\eta$-potential measure of $Y$ upon leaving $[0,k)$ under $\P_y^\beta$ is, for $y,z\in[0,k]$, given by
\begin{eqnarray}
U^{(\eta)}(y,dz)&=&\bigg(\WW{\eta}{k-y}{\beta}\frac{\dWW{\eta}{z}{\beta}}{\dWW{\eta}{k}{\beta}}-\WW{\eta}{z-y}{\beta}\bigg)dz\notag\\
&&+\WW{\eta}{k-y}{\beta}\frac{\WW{\eta}{0}{\beta}}{\dWW{\eta}{k}{\beta}}\delta_0({\rm d}z).\label{lebae}
\end{eqnarray}
Here and for the remainder of this section, unless otherwise stated, all derivatives of scale functions will be understood as the right limit of their densities with respect to Lebesgue measure.
Using the expression in (\ref{lebae}), we see that for $y\geq 0$,
\begin{eqnarray}
\tilde V_k(y)&=&\int_0^k f(z)\Bigg(\WW{\eta}{k-y}{\beta}\frac{\dWW{\eta}{z}{\beta}}{\dWW{\eta}{k}{\beta}}-\WW{\eta}{z-y}{\beta}\Bigg)\,{\rm d}z\notag\\
&&-\WW{\eta}{k-y}{\beta}\frac{\WW{\eta}{0}{\beta}}{\dWW{\eta}{k}{\beta}}\label{c5_value_functions_1}.
\end{eqnarray}
If $(q,\beta)\in\mathcal{A}$ is such that $q>\phi(\beta)$ the result follows by setting $\eta=q-\phi(\beta)$. Hence, the remaining case is when $q=0$ and $\beta<0$ (and hence $\phi(\beta)>0$). In this case, note that by Lemma~\ref{c5_adm} we have for any $w\in U:=\{z\in\mathbb{C}:\mathfrak{Re}(z)>-\phi(\beta)\}$,
\begin{equation*}
\vert\E_y^\beta[\int_0^{\tau_k} {\rm e}^{-wt}f(Y_t)\,{\rm d}t]\vert\leq\E_y^\beta[\int_0^{\tau_k} {\rm e}^{\phi(\beta)t}\,{\rm d}t]<\infty.
\end{equation*}
Now define for $w\in U$ the functions
\begin{eqnarray*}
&&g(w):=\E_y^\beta[\int_0^{\tau_k} {\rm e}^{-wt}f(Y_t)\,{\rm d}t]\quad\text{and}\\
&&g_n(w):=\E_y^\beta[\int_0^{\tau_k} {\rm e}^{-wt}f(Y_t)\,{\rm d}t1_{\{\tau_k\leq n\}}],\qquad n\geq 0.
\end{eqnarray*}
The functions $g_n$ are analytic in $U$ since one can differentiate under the integral sign. Moreover, for $w\in U$ we have the estimate
\begin{equation*}
\vert g(w)-g_n(w)\vert \leq\E_y^\beta[\int_0^{\tau_k} {\rm e}^{\phi(\beta)t}\,{\rm d}t1_{\{\tau_k>n\}}]
\end{equation*}
which together with the fact that the right-hand side tends to zero as $n\uparrow\infty$ implies that $g_n$ converges uniformly to $g$ in $U$. Thus, Weierstrass' theorem shows that $g$ is analytic in $U$. Next, we deal with the right-hand side of~\eqref{c5_value_functions_1}. From the series representation of $\W{q}{x}$ provided in the proof of Lemma 3.6 in~\cite{KuzKypRiv}, it is possible to show that (after some work) the right-hand side of~\eqref{c5_value_functions_1} is also analytic (on the whole of $\mathbb{C}$). By the identity theorem it then follows that~\eqref{c5_value_functions_1} holds for $\eta\in U$, in particular for real $\eta$ such that $\eta>-\phi(\beta)$. Finally, to obtain the result for $\eta=-\phi(\beta)$, take limits on both sides of~\eqref{c5_value_functions_1} and use dominated convergence on the left-hand side and analyticity on the right-hand side. This completes the proof.
\end{proof}

Having this semi-explicit form for $V_k$, the next step is to find the ``good'' threshold $k>0$. This is done using the principle of smooth or continuous fit (cf.~\cite{mikhalevich,pes_shir,peskir}) which suggests choosing $k$ such that $\lim_{y\uparrow k}V^\prime_k(y)=0$ if $\Xi$ is of unbounded variation and $\lim_{y\uparrow k}V_k(y)=0$ if $\Xi$ is of bounded variation. Note that, although the smooth or continuous fit condition is not necessarily part of the general theory of optimal stopping, it is imposed by the ``rule of thumb'' outlined in Section 7 of~\cite{some_remarks}.\\
\indent First assume that~$\Xi$ is of unbounded variation. In that case, we know that scale functions are continuously differentiable on $(0,\infty)$.  Using~\eqref{c5_scale5} and~\eqref{c5_continuityatorigin}, it follows that
\begin{eqnarray*}
V^\prime_k(y)&=&\int_y^kf(z)\dWW{q-\phi(\beta)}{z-y}{\beta}\,{\rm d}z-\frac{\dWW{q-\phi(\beta)}{k-y}{\beta}}{\dWW{q-\phi(\beta)}{k}{\beta}}\int_0^kf(z)\dWW{q-\phi(\beta)}{z}{\beta}\,{\rm d}z.
\end{eqnarray*}
Letting $y$ tend to $k$ yields
\begin{equation}
0=\lim_{y\uparrow k}\frac{\dWW{q-\phi(\beta)}{k-y}{\beta}}{\dWW{q-\phi(\beta)}{k}{\beta}}\int_0^kf(z)\dWW{q-\phi(\beta)}{z}{\beta}\,{\rm d}z.\label{c5_conc}
\end{equation}
Now note that by~\eqref{c5_scale5} and~\eqref{c5_derivativeatorigin} we have
\begin{equation*}
\lim_{y\uparrow k}\dWW{q-\phi(\beta)}{k-y}{\beta}=\lim_{y\uparrow k}{\rm e}^{-\beta(k-y)}(\dW{q}{k-y}-\beta\W{q}{k-y})\in(0,\infty].
\end{equation*}
Similarly, $\dWW{q-\phi(\beta)}{k}{\beta}={\rm e}^{-\beta k}(\dW{q}{k}-\beta\W{q}{k})$ which is clearly positive if $\beta<0$. If $\beta>0$, this is still true because $\dW{q}{z}/\W{q}{z}>\Phi(q)$
 for $z>0$ and $\Phi(q)>\beta$. In view of~\eqref{c5_conc}, we are forced to conclude that
\begin{equation*}
\int_0^kf(z)\dWW{q-\phi(\beta)}{z}{\beta}\,{\rm d}z=0.
\end{equation*}
Similarly, if $\Xi$ is of bounded variation, we get
\begin{equation*}
0=\frac{\WW{q-\phi(\beta)}{0}{\beta}}{\dWW{q-\phi(\beta)}{k}{\beta}}\bigg(\int_0^kf(z)\dWW{q-\phi(\beta)}{z}{\beta}\,{\rm d}z-\WW{q-\phi(\beta)}{0}{\beta}\bigg)
\end{equation*}
and hence, using~\eqref{c5_scale5} and~\eqref{c5_continuityatorigin}, we infer
\begin{equation}
\int_0^kf(z)\dWW{q-\phi(\beta)}{z}{\beta}\,{\rm d}z=\W{q}{0}.\label{c5_cond_k}
\end{equation}
Summing up, irrespective of the path variation of $\Xi$, we expect the optimal $k>0$ to solve~\eqref{c5_cond_k} and therefore we need to investigate the equation more closely.
\begin{lem}\label{c5_root}
The equation
\begin{equation}
h(k):=\int_0^kf(z)\dWW{q-\phi(\beta)}{z}{\beta}\,{\rm d}z-\W{q}{0}=0\label{c5_k_eq}
\end{equation}
has a unique solution $k^*$ on $(0,\infty)$. In particular, $k^*>\log(2)/\Phi(q)$.
\end{lem}

\begin{proof}
Using~\eqref{c5_scale5}, it follows that
\begin{equation*}
h^\prime(k)=f(k){\rm e}^{-\beta k}(\dW{q}{k}-\beta\W{q}{k})
\end{equation*}
for $k>0$. If $(q,\beta)\in\mathcal{A}$ such that $\beta>0$, then $\Phi(q)>\beta$ and, using~\eqref{c5_scale5},
\begin{equation*}
\frac{\dW{q}{z}}{\W{q}{z}}=\frac{W^\prime_{\Phi(q)}(z)}{W_{\Phi(q)}(z)}+\Phi(q)>\Phi(q)>\beta
\end{equation*}
for $z>0$. Therefore, we see that $h^\prime(k)<0$ on $(0,k_0)$, $h^\prime(k_0)=0$ and $h^\prime(k)>0$ on $(k_0,\infty)$, where $k_0=\log(2)/\Phi(q)$. The same is of course true if $(q,\beta)\in\mathcal{A}$ and $\beta<0$. Additionally, it holds that $\lim_{k\uparrow\infty}h(k)>0$. Indeed, let $z_0>k_0$ such that $f(z)\geq 1/2$ for $z\geq z_0$ and hence for $k>z_0$,
\begin{eqnarray*}
h(k)&=&h(k_0)+\int_{k_0}^kf(z)\dWW{q-\phi(\beta)}{k}{\beta}\,{\rm d}z-\W{q}{0}\\
&\geq& h(k_0)+\frac{1}{2}\int_{z_0}^k\dWW{q-\phi(\beta)}{z}{\beta}\,{\rm d}z-\W{q}{0}\\
&=&h(k_0)+\frac{1}{2}({\rm e}^{-\beta k}\W{q}{k}-{\rm e}^{-\beta z_0}\W{q}{z_0})-\W{q}{0},
\end{eqnarray*}
where in the last equality we have used~\eqref{c5_scale5}. Again by~\eqref{c5_scale5}, $\W{q}{k}={\rm e}^{\Phi(q)k}W_{\Phi(q)}(k)$ which together with the fact that $\Phi(q)>\beta$ implies that the right-hand side tends to infinity as $k\uparrow\infty$. Combining this with the fact that $
f(-1)= 0$ and the intermediate value theorem shows that there is a unique $k^*>k_0$ such that $h(k^*)=0$. This completes the proof.
\end{proof}

We are now in a position to formulate our main result of this section.
\begin{theorem}\label{c5_main_res}
The solution to~\eqref{c5_problem7} is given by
\begin{eqnarray}
V^*(y)&=&-\int_y^{k^*}f(z){\rm e}^{-\beta(z-y)}\W{q}{z-y}\,{\rm d}z,\qquad y\geq 0,\label{c5_defvalue}
\end{eqnarray}
with optimal stopping time $\tau_{k^*}$, where $k^*$ is as in Lemma~\ref{c5_root}.
\end{theorem}

\begin{proof}
Let $V$ be defined as the right-hand side of~\eqref{c5_defvalue}. It is enough to check the following conditions:
{\renewcommand{\theenumi}{\roman{enumi}}
\renewcommand{\labelenumi}{(\theenumi)}
\begin{enumerate}
\item\label{c5_cond1} $V(y)\leq 0$ for all $y\geq 0$;
\item\label{c5_cond2} the process
\begin{equation*}
{\rm e}^{-(q-\phi(\beta))t}V(Y_t)+\int_0^t{\rm e}^{-(q-\phi(\beta))u}f(Y_u)\,{\rm d}u,\qquad t\geq 0,
\end{equation*}
is a $\P_y^\beta$-submartingale for all $y\geq 0$.
\end{enumerate}}
To see why these are sufficient conditions, note that~\eqref{c5_cond1} and~\eqref{c5_cond2} together with Fatou's lemma in the second inequality and Doob's stopping theorem in the third inequality show that for $\tau\in\mathcal{M}$,
\begin{eqnarray*}
\E_y^\beta[\int_0^\tau {\rm e}^{-(q-\phi(\beta))t}f(Y_u)\,{\rm d}u]&\geq&\E_y^\beta[V(Y_\tau)+\int_0^\tau {\rm e}^{-(q-\phi(\beta))t}f(Y_u)\,{\rm d}u]\\
&\geq&\limsup_{t\uparrow\infty}\E_y^\beta[V(Y_{t\wedge\tau})+\int_0^{t\wedge\tau} {\rm e}^{-(q-\phi(\beta))u}f(Y_u)\,{\rm d}u]\\*
&\geq&V(y).
\end{eqnarray*}
Since these inequalities are all equalities for $\tau=\tau_{k^*}$ the result follows.

The remainder of this proof is devoted to checking conditions~\eqref{c5_cond1} and~\eqref{c5_cond2}.\\

\noindent\textbf{Verification of condition~\eqref{c5_cond1}:} Recall that $k^*>k_0=\log(2)/\Phi(q)$ and that $f(z)\leq0$ on $(0,k_0]$ and $f(z)>0$ on $(k_0,\infty)$. It follows that $\tau_{k^*}\geq\tau_{k_0}$ and hence, using the strong Markov property, we see that
\begin{eqnarray*}
V(y)&=&\E_y^\beta[\int_0^{\tau_{k_0}} {\rm e}^{-(q-\phi(\beta))t}f(Y_t)\,{\rm d}t]+\E_y^\beta[\int_{\tau_{k_0}}^{\tau_{k^*}} {\rm e}^{-(q-\phi(\beta))t}f(Y_t)\,{\rm d}t]\\
&=&\E_y^\beta[\int_0^{\tau_{k_0}} {\rm e}^{-(q-\phi(\beta))t}f(Y_t)\,{\rm d}t]+\E_y^\beta[{\rm e}^{-(q-\phi(\beta))\tau_{k_0}}V(Y_{\tau_{k_0}})]\\
&\leq&0,
\end{eqnarray*}
where the last inequality follows from the fact that $f(z)\leq 0$ on $(0,k_0]$ and $V(y)\leq 0$ on $[k_0,\infty)$. This completes the proof of~\eqref{c5_cond1}.\\

\noindent\textbf{Verification of condition~\eqref{c5_cond2}:} The proof of this appeals to standard techniques and, hence, we only outline the main steps and omit the details.\\
\indent As for a first step, one may use the Markov property to show that the process
\begin{equation*}
Z_t:={\rm e}^{-(q-\phi(\beta))(t\wedge\tau_{k^*})}V(Y_{t\wedge\tau_{k^*}})+\int_0^{t\wedge\tau_{k^*}} {\rm e}^{-(q-\phi(\beta))u}f(Y_u)\,{\rm d}u,\qquad t\geq 0,
\end{equation*}
is a $\P_y^\beta$-martingale for $0<y<k^*$. Indeed, for $t\geq 0$, the strong Markov property gives
\begin{eqnarray*}
\E_y^\beta[Z_{\tau_{k^*}}\vert\mathcal{F}_t]&=&Z_{\tau_{k^*}}1_{\{\tau_{k^*}<t\}}+\E_y^\beta[Z_{\tau_{k^*}}\vert\mathcal{F}_t]1_{\{t\leq\tau_{k^*}\}}\\
&=&Z_{\tau_{k^*}}1_{\{\tau_{k^*}<t\}}+\int_0^t{\rm e}^{-(q-\phi(\beta))u}f(Y_u)\,{\rm d}u1_{\{t\leq\tau_{k^*}\}}\\
&&+{\rm e}^{-(q-\phi(\beta))t}V(Y_t)1_{\{t\leq\tau_{k^*}\}}\\
&=&Z_{t\wedge\tau_{k^*}}
\end{eqnarray*}
from which the desired martingale property follows.\\
\indent As for a second step, use Doob's optional stopping theorem to deduce that for $0<k<k^*$ the process ${\rm e}^{-(q-\phi(\beta)t)(t\wedge\tau_k)}V(Y_{t\wedge\tau_k})$, $t\geq 0$, is a $\P_y^\beta$-martingale for \mbox{$0\leq y<k$}. Using this in conjunction with the appropriate version of It\^ o's formula (cf.~Theorem~71, Chapter~IV of~\cite{protter}) implies that
\begin{equation}
(\hat\Gamma^\beta V)(y)-(q-\phi(\beta))V(y)=-f(y),\qquad y\in[0,k^*),\label{c5_gen_eq}
\end{equation}
where $\hat\Gamma^\beta$ is the generator of $-\Xi$ under $\tilde\P^\beta$.\\
\indent Finally, applying the appropriate version of It\^o's formula one more time to the process ${\rm e}^{-(q-\phi(\beta)t)}V(Y_t)$, $t\geq 0$, and using~\eqref{c5_gen_eq} shows that
\begin{equation*}
{\rm e}^{-(q-\phi(\beta))t}V(Y_t)+\int_0^t{\rm e}^{-(q-\phi(\beta))u}f(Y_u)\,{\rm d}u,\qquad t\geq 0,
\end{equation*}
is a $\P_y^\beta$-submartingale for all $y\geq 0$. This finishes the sketch of the proof of~\eqref{c5_cond2}.
\end{proof}

\section{Proofs of main results}\label{c5_proofs}

\begin{proof}[Proof of Theorem~\ref{c5_main_res_1}]
The result follows by Lemma~\ref{c5_reduction_1} (and what was said just after it) and Theorem~\ref{c5_main_res}. Specifically, using Theorem~\ref{c5_main_res} with $\Xi$ equal to $\xi$ unkilled, $\beta=\alpha$ (noting in particular that $\phi(\alpha)<q$ by assumption), $y=\log(s/x)$ and then setting $K^*:={\rm e}^{-k^*}$ gives
\begin{eqnarray*}
v(x,s)&=&-x^\alpha\int_{\log(s/x)}^{-\log(K^*)}(1-2{\rm e}^{-\Phi(q)u})\WW{q-\phi(\alpha)}{u-\log(s/x)}{\alpha}\,{\rm d}u\\
&=&-x^\alpha\int^x_{K^*s}z^{-1}(1-2{\rm e}^{-\Phi(q)\log(s/z)})\WW{q-\phi(\alpha)}{\log(x/z)}{\alpha}\,{\rm d}z,
\end{eqnarray*}
where in the second equality we changed variables according to $u=\log(s/z)$. The expression for $v(x,s)$ in the theorem now follows after an application of~\eqref{c5_scale5}. As for the optimal constant $K^*$, we see that $K^*$ satisfies the equation
\begin{eqnarray*}
\int_0^{\log(1/K)}(1-2{\rm e}^{-\Phi(q)z})\dWW{q-\phi(\alpha)}{z}{\alpha}\,{\rm d}z=\W{q}{0}\quad\text{on $(0,1)$}
\end{eqnarray*} 
and the proof is complete.
\end{proof}

\begin{proof}[Proof of Theorem \ref{meantheta}] Write $\Theta^{(x)}$ in place of $\Theta$ to emphasise the dependency on the initial position $X_0 = x>0$. Self-similarity, and in particular the Lamperti transform,  implies that
\begin{equation}
x^{-\alpha}\Theta^{(x)} = \int_0^{{G}}{\rm e}^{\alpha \xi_t}\,{\rm d}t,
\label{that}
\end{equation}
where ${G} = \sup\{t>0 : \xi_t = \overline{\xi}_{\mathbf{e}}\}$. It follows that $E_x(\Theta)<\infty$ for all $x>0$ if and only if $E_1(\Theta)<\infty$. Following standard excursion theory, cf. Chapter 6 of \cite{kyprianou} or Chapter VI  of \cite{bertoin_book}, making particular use of the fact that the ladder height process of a (killed) spectrally negative L\'evy process is a (killed) unit drift, we have 
\begin{equation}
E_1(\Theta) = E_1\left[\sum_{t<\chi} {\rm e}^{\alpha t}\int_0^{\varsigma_t}{\rm e}^{-\alpha\epsilon_t(s)}{\rm d}s\right] + E_1\left[\int_0^{\mathbf{e}} {\rm e}^{\alpha t} \mathbf{1}_{(\overline{\xi}_t - \xi_t = 0)}{\rm d}t\right],
\label{2exp}
\end{equation}
such that  the sum is taken over a Poisson point process of excursions of $\overline\xi -\xi$ from zero,  $\{(t, \epsilon_t): t\in \mathcal{I}\}$,  where $\mathcal{I}$ is the index set of the point process,  $\epsilon_t :  = \{\epsilon_t(s) : s\leq \varsigma_t\}$ and $\varsigma_t$ is the excursion length of $\epsilon_t$; moreover, $\chi : = \inf\{t> 0 : \varsigma_t = \infty\}$ in the case that $\mathbf{e} = \infty$ and, otherwise, $\chi : = \inf\{t>0 : \varsigma_t >\mathbf{e}^{(t)}\}$, where, for each excursion indexed by $t>0$, $\mathbf{e}^{(t)}$ is an independent copy of the exponential random variable $\mathbf{e}$. Write ${n}$ for the intensity measure of this Poisson point process of excursions. For the case of a spectrally negative L\'evy process, it is well known that $\chi$ is exponentially distributed with parameter $\Phi(q)$. Note that $\Phi(q)$ is strictly positive if $\xi$ drifts to $-\infty$ or $\psi(0)<0$, i.e. the process $\xi$ is killed. If we write $\texttt{a} = \lim_{p\to\infty}\Phi(p)/p$, then it is also known (cf. the computations in Section 6.3 of \cite{kyprianou}) that the second expectation on the right-hand side of (\ref{2exp}) is equal to $\texttt{a}E_1[\int_0^{\mathbf{e}}{\rm e}^{\alpha t} {\rm d}\ell_t]$ where $\{\ell_t: t\geq 0\}$ is the local time of $\overline\xi-\xi$ at 0. The compensation formula and the observation that $\chi = \ell_\mathbf{e}$ now tell us that
\begin{eqnarray}
E_1(\Theta)  &=& E_1 \left[\int_0^\chi {\rm e}^{\alpha t}{\rm d}t\right]{n}\left(\int_0^\varsigma e^{-\alpha\epsilon(s)}{\rm d}z; \, \varsigma<\mathbf{e}\right) + \texttt{a} E_1 \left[\int_0^\chi {\rm e}^{\alpha t}{\rm d}t\right]\notag\\
&=&\frac{1}{\alpha}E_1[{\rm e}^{\alpha \chi} -1]\left({n}\left(\int_0^\varsigma e^{-\alpha\epsilon(s)}{\rm d}z; \, \varsigma<\mathbf{e}\right) + \texttt{a}\right)
\label{finite?}
\end{eqnarray}
such that both sides are finite (resp. infinite) at the same time. We immediately see that $E_1[\exp(\alpha\chi)]<\infty$ if and only if $\Phi(q)>\alpha$, that is to say, if and only if $\psi(\alpha)<0$. Appealing to a method developed in \cite{doney} together with Theorem VI.20 and Lemma VII.7 in \cite{bertoin_book}, it can be checked that 
\begin{eqnarray*}
{n}\left(\int_0^\varsigma e^{-\alpha\epsilon(s)}{\rm d}z; \, \varsigma<\mathbf{e}\right)
&=&\lim_{\varepsilon\downarrow0} \frac{1}{\varepsilon}\mathbb{E}^{\Phi(q)}_{-\varepsilon}\left(\int_0^{\tau^+_0} {\rm e}^{\alpha \xi_s}{\rm d}s \right)\\
&=&\lim_{\varepsilon\downarrow0} {\rm e}^{\alpha \varepsilon}\frac{1}{\varepsilon}\int_{[0,\infty)} {\rm e}^{-\alpha y}{\rm d}\hat{V}_{\Phi(q)}(y)\int_{[0,\varepsilon]}{\rm e}^{\alpha z}{\rm d}V_{\Phi(q)}(z),
\end{eqnarray*}
where $\mathbb{P}^{\Phi(q)}$ is the law of $\xi$ under $P_1$ after the change of measure given in (\ref{c5_changeofmeasure}) with $v = \Phi(q)$ and  ${V}_{\Phi(q)}$ and $\hat{V}_{\Phi(q)}$ are the renewal measures of the ascending and descending ladder height processes of $(\xi, \mathbb{P}^{\Phi(q)})$ respectively. 
From, e.g. \cite{bertoin-doney}, it is known that ${\rm d}V_{\Phi(q)}(z) = {\rm d}z$ on $z\geq 0$ and 
\[
\int_{[0,\infty)} {\rm e}^{-\beta y}{\rm d}\hat{V}_{\Phi(q)}(y) = \frac{\beta}{\psi(\beta+ \Phi(q))}, \qquad \beta\geq 0.
\]
Putting the pieces together, we have 
\begin{equation}
{n}\left(\int_0^\varsigma e^{-\alpha\epsilon(s)}{\rm d}z; \, \varsigma<\mathbf{e}\right) = \frac{\alpha}{\psi(\Phi(q)+\alpha)},
\label{excursion_comp}
\end{equation}
which is finite and hence, considering this in the context of   (\ref{finite?}), the result is proved.
\end{proof}

\begin{proof}[Proof of Theorem~\ref{c5_main_res_2}]
The result follows by Lemma~\ref{c5_reduction_2} (and what was said just after it) and Theorem~\ref{c5_main_res} with $\Xi$. Specifically, using Theorem~\ref{c5_main_res} with $\Xi$ equal to $\hat\xi$ unkilled, $\beta=-\alpha$ (noting in particular that $\psi(-\alpha)$ exists by assumption with $\phi(-\alpha)<q$ if $q>0$), $y=\log(x/i)$ and then setting $\hat K^*:={\rm e}^{k^*}$ gives
\begin{eqnarray*}
\hat v(x,i)&=&-x^\alpha\int_{\log(x/i)}^{\log(\hat K^*)}(1-2{\rm e}^{-\hat\Phi(q)u})\hat W^{(q-\hat\phi(-\alpha))}_{-\alpha}(u-\log(x/i))\,{\rm d}u\\
&=&-x^\alpha\int^{\hat K^*i}_{x}z^{-1}(1-2{\rm e}^{-\hat\Phi(q)\log(z/i)})W_{-\alpha}^{(q-\hat\phi(-\alpha))}(\log(z/x))\,{\rm d}z,
\end{eqnarray*}
where in the second equality we changed variables according to $u=\log(z/i)$. The expression for $\hat v(x,i)$ in the theorem now follows after an application of~\eqref{c5_scale5}. As for the optimal constant $\hat K^*$, we see that $\hat K^*$ satisfies the equation
\begin{eqnarray*}
\int_0^{\log(K)}(1-2{\rm e}^{-\hat\Phi(q)z})\hat W^{(q-\hat\phi(-\alpha))^\prime}_{-\alpha}(z)\,{\rm d}z=\hat W^{(q)}(0)\qquad\text{on $(1,\infty)$},
\end{eqnarray*} 
thus completing the proof.
\end{proof}

\begin{proof}[Proof of Theorem \ref{thetahatfinite}]
 Note that, similarly to the proof of Theorem \ref{meantheta},  it suffices to consider the case that $x=1$. Similarly to (\ref{that}) we have, under $P_1$, 
 that 
 \[
\hat\Theta = \int_0^{\hat{G}}{\rm e}^{\alpha\xi_t}{\rm d}t,
\]
where $\hat{G} = \sup\{t>0 : \xi_t = \underline{\xi}_\mathbf{e} \}$. A similar analysis via excursion theory shows that, as long as the right hand side is finite, 
\[
E_1(\hat\Theta)  = \frac{1}{\alpha}[1-{\rm e}^{-\alpha\hat\chi}]\left(\hat{n}\left(\int_0^\varsigma e^{\alpha\epsilon(s)}{\rm d}z; \, \varsigma<\mathbf{e}\right) + \hat{\texttt{a}}\right),
\]
where $\hat{\chi}, \hat{n}$ and $\hat{\texttt{a}}$ play the same role as $\chi, n$ and $\texttt{a}$, but now for the process $\hat\xi$. Following the reasoning that leads to (\ref{excursion_comp}) we see that, whenever the right-hand side is finite,
\begin{equation}
\hat{n}\left(\int_0^\varsigma e^{\alpha\epsilon(s)}{\rm d}z; \, \varsigma<\mathbf{e}\right) = \frac{-\alpha}{\hat\psi(\hat\Phi(q)-\alpha)}.
\label{lhs}
\end{equation}
Note that, since $\hat\psi(\hat\Phi(q)) = 0$, when the right-hand side is finite, then it is positive valued. Moreover, the left-hand side of (\ref{lhs}) is a monotone function of $\alpha$ and hence we see that, when $q=0$, $E_1(\hat\Theta)<\infty$ if and only if $\alpha<\hat\Phi(0)$, i.e. $\hat\psi(\alpha)<0$.  Moreover, if $q>0$, then requiring that $\alpha<\hat\Phi(q)$, i.e. $\hat\psi(\alpha)<0$, is a sufficient condition for $E_1(\hat\Theta)<\infty$ but a necessary condition would require us to know how the exponent $\hat\psi$ behaves if its domain is extended into the negative half line. 
\end{proof}

\section{Examples}\label{c5_examples}
In this section we present two examples, one of which shows that our results are consistent with the existing literature.

\begin{cor}\label{c5_example3}
Let $X$ be a pssMp with index of self-similarity $\alpha>0$ such that its Lamperti representation is given by $\xi_t=\sigma W_t-\mu t$, $t\geq 0$, where $\sigma>0$, $\mu>0$ and $W_t$, $t\geq 0$, is a standard Brownian motion. In other words, $X$ is of type~\eqref{c5_type2} such that $\lim_{t\uparrow\infty}X_t=-\infty$. Moreover, suppose that $\alpha<2\mu/\sigma^2$ (this ensures that $X\in\mathcal{C}^1$). Then we have
\begin{eqnarray*}
&&v(x,s)=\frac{1}{\mu}\Bigg[x^\alpha\bigg(1-\bigg(\frac{K^*s}{x}\bigg)^\alpha\bigg)\bigg(\frac{1}{\alpha}+\frac{2}{\alpha}\bigg(\frac{x}{s}\bigg)^{\Phi(0)}\bigg)\\
&&-\frac{x^{\alpha}}{\alpha-\Phi(0)}\bigg(1-\bigg(\frac{K^*s}{x}\bigg)^{\alpha-\Phi(0)}\bigg)+\frac{2s^\alpha (K^*)^{\alpha+\Phi(0)}}{\alpha+\Phi(0)}\bigg(1-\bigg(\frac{K^*s}{x}\bigg)^{-\Phi(0)-\alpha}\bigg)\Bigg],
\end{eqnarray*}
where $\Phi(0)=2\mu/\sigma^2$, and $K^*$ is the unique solution to
\begin{equation*}
K^{\alpha-\Phi(0)}+\frac{2\Phi(0)-3\alpha}{\alpha}K^\alpha+\frac{2\alpha}{\alpha+\Phi(0)}K^{\alpha+\Phi(0)}-\frac{2\Phi(0)^2}{\alpha(\alpha+\Phi(0))}=0
\end{equation*}
on $(0,1)$. In particular, $K^*\in(0,2^{-1/\Phi(0)})$.
\end{cor}

\begin{proof}
It is easy to check that $\psi(\theta)=\frac{\sigma^2}{2}\theta^2-\mu\theta$, $\Phi(0)=\frac{2\mu}{\sigma^2}$ and $W^{(0)}(x)=\frac{{\rm e}^{x\Phi(0)}-1}{\mu}$. Also note that $\alpha<\Phi(0)$ by assumption. For convenience, write $k=K^*$. It now follows from Theorem~\ref{c5_main_res_1} that
\begin{eqnarray*}
v(x,s)&=&-\int_{ks}^{x}\big(1-2(z/s)^{\Phi(0)}\big)z^{\alpha-1}\frac{(x/z)^{\Phi(0)}-1}{\mu}\,{\rm d}z\\
&=&\frac{1}{\mu}\Bigg[-x^{\Phi(0)}\int_{ks}^{x}z^{\alpha-1-\Phi(0)}\,{\rm d}z+\int_{ks}^{x}z^{\alpha-1}\,{\rm d}z\\
&&+2(x/s)^{\Phi(0)}\int_{ks}^{x}z^{\alpha-1}\,{\rm d}z-2s^{-\Phi(0)}\int_{ks}^{x}z^{\alpha-1+\Phi(0)}\,{\rm d}z\Bigg]\\
&=&\frac{1}{\mu}\Bigg[-x^{\Phi(0)}\Bigg(\frac{x^{\alpha-\Phi(0)}}{\alpha-\Phi(0)}-\frac{(ks)^{\alpha-\Phi(0)}}{\alpha-\Phi(0)}\Bigg)+\frac{x^\alpha}{\alpha}-\frac{(ks)^\alpha}{\alpha}\\
&&+2(x/s)^{\Phi(0)}\Bigg(\frac{x^\alpha}{\alpha}-\frac{(ks)^\alpha}{\alpha}\Bigg)-2s^{-\Phi(0)}\bigg(\frac{x^{\alpha+\Phi(0)}}{\alpha+\Phi(0)}-\frac{(ks)^{\alpha+\Phi(0)}}{\alpha+\Phi(0)}\bigg)\Bigg]\\
&=&\frac{1}{\mu}\Bigg[\frac{x^{\alpha}}{\alpha-\Phi(0)}\bigg(\bigg(\frac{ks}{x}\bigg)^{\alpha-\Phi(0)}-1\bigg)-\frac{x^\alpha}{\alpha}\bigg(\bigg(\frac{ks}{x}\bigg)^\alpha-1\bigg)\\
&&-\frac{2s^{-\Phi(0)}x^{\alpha+\Phi(0)}}{\alpha}\bigg(\bigg(\frac{ks}{x}\bigg)^\alpha-1\bigg)-\frac{2s^\alpha k^{\alpha+\Phi(0)}}{\alpha+\Phi(0)}\bigg(\bigg(\frac{ks}{x}\bigg)^{-\Phi(0)-\alpha}-1\bigg)\Bigg].
\end{eqnarray*}
Adding the second and third term gives
\begin{eqnarray*}
v(x,s)&=&\frac{1}{\mu}\Bigg[x^\alpha\bigg(1-\bigg(\frac{ks}{x}\bigg)^\alpha\bigg)\bigg(\frac{1}{\alpha}+\frac{2}{\alpha}\bigg(\frac{x}{s}\bigg)^{\Phi(0)}\bigg)\\
&&-\frac{x^{\alpha}}{\alpha-\Phi(0)}\bigg(1-\bigg(\frac{ks}{x}\bigg)^{\alpha-\Phi(0)}\bigg)+\frac{2s^\alpha k^{\alpha+\Phi(0)}}{\alpha+\Phi(0)}\bigg(1-\bigg(\frac{ks}{x}\bigg)^{-\Phi(0)-\alpha}\bigg)\Bigg].
\end{eqnarray*}Next, let us derive the equation for $K^*$. Using~\eqref{c5_scale5} and changing variables according to $u={\rm e}^z$ shows that $K^*$ is the unique root of
\begin{equation}
\int_1^{1/K}u^{-\alpha-1}(1-2u^{-\Phi(0)})(\Phi(0)u^{\Phi(0)}-\alpha u^{\Phi(0)}+\alpha)\,{\rm d}u=0\qquad\text{on $(0,1)$}.\label{c5_equation_for_k}
\end{equation}
Solving the integral and rearranging gives the claim.
\end{proof}

\begin{cor}\label{c5_example4}
Let $X$ be a pssMp with index of self-similarity $\alpha>0$ such that its Lamperti representation is given by $\xi_t=\sigma W_t+\mu t$, $t\geq 0$, where $\sigma>0$, $\mu>0$ and $W_t$, $t\geq 0$, is a standard Brownian motion. In other words, $X$ is of type~\eqref{c5_type1} such that $\lim_{t\uparrow\infty}X_t=\infty$.
\begin{enumerate}
\item If $\alpha\neq2\mu/\sigma^2$, we have
\begin{eqnarray*}
&&\hat v(x,i)=\frac{1}{\mu}\Bigg[x^\alpha\bigg(\bigg(\frac{\hat K^*i}{x}\bigg)^\alpha-1\bigg)\bigg(\frac{1}{\alpha}+\frac{2}{\alpha}\bigg(\frac{i}{x}\bigg)^{\hat\Phi(0)}\bigg)\label{c5_v_1}\\
&&-\frac{x^{\alpha}}{\alpha+\hat\Phi(0)}\bigg(\bigg(\frac{\hat K^*i}{x}\bigg)^{\alpha+\hat\Phi(0)}-1\bigg)-\frac{2i^\alpha (\hat K^*)^{\alpha-\hat\Phi(0)}}{\hat\Phi(0)-\alpha}\bigg(\bigg(\frac{\hat K^*i}{x}\bigg)^{\hat\Phi(0)-\alpha}-1\bigg)\Bigg],\notag
\end{eqnarray*}
where $\hat\Phi(0)=2\mu/\sigma^2$, and $\hat K^*$ is the unique solution to
\begin{equation*}
K^{\hat\Phi(0)+\alpha}-\frac{3\alpha+2\hat\Phi(0)}{\alpha}K^\alpha+\frac{2\alpha}{\alpha-\hat\Phi(0)}K^{\alpha-\hat\Phi(0)}-\frac{2\hat\Phi(0)^2}{\alpha(\alpha-\hat\Phi(0))}=0
\end{equation*}
on $(1,\infty)$. In particular, $\hat K^*>2^{1/\hat\Phi(0)}$.
\item If $\alpha=2\mu/\sigma^2$, we have
\begin{eqnarray*}
\hat v(x,i)&=&\frac{1}{\mu}\Bigg[x^\alpha\bigg(\frac{1}{\alpha}+\frac{2}{\alpha}\bigg(\frac{i}{x}\bigg)^\alpha\bigg)\bigg(\bigg(\frac{\hat K^*i}{x}\bigg)^\alpha-1\bigg)\label{c5_v_2}\\*
&&-\frac{x^2}{2\alpha}\bigg(\bigg(\frac{\hat K^*i}{x}\bigg)^{2\alpha}-1\bigg)-2i^\alpha\log(\hat K^*i/x)\Bigg],\notag
\end{eqnarray*}
and $\hat K^*$ is the unique solution to
\begin{equation*}
K^{2\alpha}-5K^\alpha+2\alpha\log(K)+4=0
\end{equation*}
on $(1,\infty)$. In particular, $\hat K^*>2^{1/\hat\Phi(0)}$.
\end{enumerate}
\end{cor}

\begin{proof}
Clearly, $-\xi_t=\sigma W_t-\mu t$ and it is straightforward to check that $\hat\psi(\theta)=\frac{\sigma^2}{2}\theta^2-\mu\theta$, $\hat\Phi(0)=\frac{2\mu}{\sigma^2}$ and $\hat W^{(0)}(x)=\frac{{\rm e}^{x\hat\Phi(0)}-1}{\mu}$. We derive the result for $\alpha\neq\hat\Phi(0)$, the case when $\alpha=\hat\Phi(0)$ is similar and we omit the details. For convenience, write $k=\hat K^*$. By Theorem~\ref{c5_main_res_2} we have
\begin{eqnarray*}
\hat v(x,i)&=&-\int_x^{ki}\big(1-2(i/z)^{\hat\Phi(0)}\big)z^{\alpha-1}\frac{(z/x)^{\hat\Phi(0)}-1}{\mu}\,{\rm d}z\\
&=&\frac{1}{\mu}\Bigg[-x^{-\hat\Phi(0)}\int_x^{ki}z^{\alpha-1+\hat\Phi(0)}\,{\rm d}z+\int_x^{ki}z^{\alpha-1}\,{\rm d}z\\
&&+2(i/x)^{\hat\Phi(0)}\int_x^{ki}z^{\alpha-1}\,{\rm d}z-2i^{\hat\Phi(0)}\int_x^{ki}z^{\alpha-1-\hat\Phi(0)}\,{\rm d}z\Bigg]\\
&=&\frac{1}{\mu}\Bigg[-x^{-\hat\Phi(0)}\Bigg(\frac{(ki)^{\alpha+\hat\Phi(0)}}{\alpha+\hat\Phi(0)}-\frac{x^{\alpha+\hat\Phi(0)}}{\alpha+\hat\Phi(0)}\Bigg)+\frac{(ki)^\alpha}{\alpha}-\frac{x^\alpha}{\alpha}\\
&&+2(i/x)^{\hat\Phi(0)}\Bigg(\frac{(ki)^\alpha}{\alpha}-\frac{x^\alpha}{\alpha}\Bigg)-2i^{\hat\Phi(0)}\bigg(\frac{(ki)^{\alpha-\hat\Phi(0)}}{\alpha-\hat\Phi(0)}-\frac{x^{\alpha-\hat\Phi(0)}}{\alpha-\hat\Phi(0)}\bigg)\Bigg]\\
&=&\frac{1}{\mu}\Bigg[\frac{-x^{\alpha}}{\alpha+\hat\Phi(0)}\bigg(\bigg(\frac{ki}{x}\bigg)^{\alpha+\hat\Phi(0)}-1\bigg)+\frac{x^\alpha}{\alpha}\bigg(\bigg(\frac{ki}{x}\bigg)^\alpha-1\bigg)\\
&&+\frac{2i^{\hat\Phi(0)}x^{\alpha-\hat\Phi(0)}}{\alpha}\bigg(\bigg(\frac{ki}{x}\bigg)^\alpha-1\bigg)-\frac{2i^\alpha k^{\alpha-\hat\Phi(0)}}{\hat\Phi(0)-\alpha}\bigg(\bigg(\frac{ki}{x}\bigg)^{\hat\Phi(0)-\alpha}-1\bigg)\Bigg].
\end{eqnarray*}
Adding the second and third term gives
\begin{eqnarray*}
\hat v(x,i)&=&\frac{1}{\mu}\Bigg[x^\alpha\bigg(\bigg(\frac{ki}{x}\bigg)^\alpha-1\bigg)\bigg(\frac{1}{\alpha}+\frac{2}{\alpha}\bigg(\frac{i}{x}\bigg)^{\hat\Phi(0)}\bigg)\\
&&-\frac{x^{\alpha}}{\alpha+\hat\Phi(0)}\bigg(\bigg(\frac{ki}{x}\bigg)^{\alpha+\hat\Phi(0)}-1\bigg)-\frac{2i^\alpha k^{\alpha-\hat\Phi(0)}}{\hat\Phi(0)-\alpha}\bigg(\bigg(\frac{ki}{x}\bigg)^{\hat\Phi(0)-\alpha}-1\bigg)\Bigg].
\end{eqnarray*}
Next, let us derive the equation for $\hat K^*$. Using~\eqref{c5_scale5} and changing variables according to $u={\rm e}^z$ shows that $\hat K^*$ has to satisfy the equation
\begin{eqnarray*}
\int_1^Ku^{\alpha-1}(1-2u^{-\hat\Phi(0)})(\alpha u^{\hat\Phi(0)}-\alpha+\hat\Phi(0)u^{\hat\Phi(0)})\,{\rm d}u=0\qquad\text{on $(1,\infty)$.}
\end{eqnarray*}
Solving the integral and rearranging gives the claim.
\end{proof}

\begin{rem}
\rm
Note that in contrast to Corollary~\ref{c5_example3}, in Corollary~\ref{c5_example4} there is no condition required to ensure that $X\in\hat{\mathcal{C}}^1$, since in this case $X$ is of type~\eqref{c5_type1} and then the only requirement is that the Laplace exponent of the Lamperti transformation of $X$ exists. This is clearly the case in Corollary~\ref{c5_example4}.
\end{rem}

\begin{rem}\label{c5_example5}
\rm
If $X$ is a $d$-dimensional Bessel process with $d>2$, then $X$ is a pssMp with index of self-similarity $\alpha=2$ and of type~\eqref{c5_type1} with $\lim_{t\uparrow\infty}X_t=\infty$. It is known that its Lamperti representation is given by $\xi_t=W_t+\frac{({\rm d}-2)}{2}t$. Setting $\sigma=1$ and $\mu=\frac{d-2}{2}$ in Corollary~\ref{c5_example4}, one recovers Theorem 4 of~\cite{GloHulPes}. In particular, if $d=3$ one sees that $\hat K^*$ is the unique solution to
\begin{equation*}
K^3-4K^2+4K-1=(K-1)(K^2-3K+1)=0
\end{equation*} 
on $(1,\infty)$. Solving this equation shows that $\hat K^*=(3+\sqrt{5})/2$. The corresponding optimal stopping time can then be expressed as
\begin{equation*}
\hat\tau^*=\inf\{t\geq0:X_t\geq\hat K^*(i\wedge\underline X_t)\}=\inf\{t\geq0:(X_t-(i\wedge\underline X_t))/(i\wedge\underline X_t)\geq \varphi\},
\end{equation*}
where $\varphi:=\hat K^*-1$ is the golden ratio. This was first observed and proved in~\cite{GloHulPes}. 
\end{rem}

\begin{rem}\rm
Note that according to Theorem \ref{thetahatfinite}, $E_x(\hat\Theta)<\infty$ for all $x>0$, if and only if $\hat\psi(2) = (2^2)/2 - 2\times(d-2)/2<0$. That is to say, $\hat\Theta$ has finite mean if and only if $d>4$. This agrees with what is already known in the literature. See for example Lemma 1 of \cite{shi}.
\end{rem}

\end{document}